\documentclass[journal,twoside,web]{ieeecolor}
\usepackage{generic}
\usepackage{cite}
\usepackage{amsmath,amssymb,amsfonts}
\usepackage{algorithmic}
\usepackage{graphicx}
\usepackage{textcomp}
\usepackage{xspace}                 

\usepackage{amsthm} 
\usepackage{epsfig}                 
\usepackage{enumerate}
\usepackage{cite}
\usepackage[justification=centering]{caption}
\graphicspath{{./figs}}
\usepackage{apptools}
\usepackage{chngcntr}

\definecolor{dccol}{RGB}{0,0,255}
\definecolor{njccol}{RGB}{255,99,0}
\definecolor{jsfcol}{RGB}{0,128,0}

\definecolor{mxj}{RGB}{10, 140, 148}

\newcommand{\rank}{\operatorname{rank}}
\newtheorem{theorem}{Theorem}[section]
\newtheorem{lemma}[theorem]{Lemma}
\newtheorem{corollary}[theorem]{Corollary}
\newtheorem{proposition}[theorem]{Proposition} 
\begin{document}
\title{\LARGE \bf Optimal Estimation with Sensor Delay}
\author{D.~Cao$^{1}$, N.~J.~Cowan$^{1}$, and J.~S.~Freudenberg$^{2}$
\thanks{*This work is supported by NSF Awards 1825489 and 1825931 and NIH Award R01-HD040289.}%
\thanks{$^{1}$D.~Cao and N.~J.~Cowan are with the Department of
    Mechanical Engineering, Johns Hopkins University, Baltimore MD
    21218 USA.  {\tt <dcao9, ncowan>@jhu.edu}}%
\thanks{$^{2}$J.~S.~Freudenberg is with the Department of Electrical
    Engineering and Computer Science, University of Michigan, Ann
    Arbor MI 48109 USA.  {\tt jfr@umich.edu}}%
}
\maketitle 
%
\begin{abstract} Given a plant subject to delayed sensor measurement,
  there are several approaches to compensate for the delay.
An obvious approach is to address this problem in state space, where
the $n$-dimensional plant state is augmented by an $N$-dimensional
(Pad\'e) approximation to the delay, affording (optimal) state
estimate feedback vis-\`a-vis the separation principle. Using this
framework, we show: (1) Feedback of the estimated plant states
partially inverts the delay; (2) The optimal (Kalman) estimator
decomposes into $N$ (Pad\'e) uncontrollable states, and the
remaining $n$ eigenvalues are the solution to a reduced-order Kalman
filter problem. Further, we show that the tradeoff of estimation
error (of the full state estimator) between plant disturbance and
measurement noise, only depends on the reduced-order Kalman filter
(that can be constructed independently of the delay); (3) A subtly modified
  version of this state-estimation-based control scheme bears close
  resemblance to a Smith predictor. This modified state-space approach
  shares several limitations with its Smith predictor analog
  (including the inability to stabilize most unstable plants),
  limitations that are alleviated when using the unmodified state
  estimation framework.
 
\end{abstract}


\section{Introduction}
\label{sec:intro}


Time delay is ubiquitous in chemical processes, biological systems, and
industrial applications. It is widely known that in biological systems 
  delay is inevitable in the sensory feedback simply because neural
transduction of information is slow
\cite{more2018scaling,madhavsynergy2020}. For example, visuomotor
delays introduce about 110-160~ms into the feedback
loop~\cite{franklin2008specificity,haith2016independence,zimmetcerebellar2020}. However,
despite these long latencies, humans and other animals manage smooth
and accurate movements with ease. So a fundamental question in
neuroscience is how the brain compensates for such delays
\cite{more2018scaling,zimmetcerebellar2020,susilaradeya2019extrinsic,crevecoeur2019filtering}. It
is believed that our nervous system could provide predictive
estimation and control through internal models of the plant, sensor
dynamics, and delay, in order to compensate for the delayed feedback.

Two common approaches for model-based delay compensation are
Smith-Predictor-like and state-observer-based controllers. Smith
Predictors rely on an accurate model of plant and delay and then, save
for (most) unstable plants, the controller can be designed without
consideration of the time delay \cite{Smith57}. There are many
improvements of the basic Smith Predictor architecture, most notably
those that could extend to unstable plants
\cite{garcia2006control, sanz2018generalized}. The state observer based method can also
compensate for time delay in the feedback loop \cite{crevecoeur2019filtering,lima2018robust}. 

Several studies have suggested that the human predictive
controller may be modeled as a Smith
Predictor\cite{miall1993cerebellum, zimmetcerebellar2020, tolu2020cerebellum}, while other studies examine
state-observer approaches
\cite{lima2018robust,susilaradeya2019extrinsic,crevecoeur2019filtering}.
Indeed, there are structural similarities between
  these two approaches; for example, Mirkin and Raskin~\cite{mirkin2003every} showed that for continuous-time systems every stabilizing time-delay
  controller has an observer--predictor-based structure.
  For discrete-time systems with input
  delay, Mirkin and Zanutto
  \cite{mirkin2021dead} showed that the discrete equivalent of the observer-predictor
  architecture can be derived via classical state-feedback and
  observer design. They further characterized the closed loop eigenvalues in the event that the state feedback is LQR optimal by showing that they   either lie at the origin or arise from   a lower order LQR problem that does not involve the delay.

In the present paper, we study continuous-time systems with a delay at the output that we model using a Pad\'{e} approximation~\cite{natori2012design,probst2010using,Franklin}, and consider a state feedback that depends only on the plant states combined with an estimator for both plant and Pad\'e states. We show in Section~\ref{sec:continuous} that
the poles of the Pad\'e approximation will appear as transmission zeros of
the observer-based compensator,  thus providing phase lead that partially masks the  lag contributed by the delay. We also describe a
tradeoff filter between the response of the estimation error to
measurement noise, on the one hand, and disturbances and mismatch between
the Pad\'e approximation and the time delay, on the other.  In Section~\ref{sec:optimal}, we assume that the observer is a steady-state Kalman filter and    show that the optimal
estimator has uncontrollable eigenvalues equal to those of the Pad\'e
approximation. The remaining  eigenvalues may be found from a lower order estimation problem that ignores the time
delay, and   the tradeoff filter is also  independent of the   delay. We exploit the special structure of the optimal estimator in Section~\ref{sec:alternate} by proposing an alternate control architecture together with conditions under which it is stabilizing. This alternate control architecture has properties that are remarkably similar to those of the Smith predictor.  We explore these similarities in
Section~\ref{sec:Smith}, and discuss further research directions and
potential connections to neuroscience in
Section~\ref{sec:conclusions}. Appendix~\ref{sec:continuous_outputdelay} contains proofs of some of the technical results from the body of the paper. In Appendix~\ref{sec:continuous_inputdelay} we present the counterpart to the results of Section~\ref{sec:optimal} for continuous-time systems with a delay at the plant input. We study discrete-time systems with an output delay in Appendices~\ref{sec:discrete_feedback} and~\ref{sec:discrete_optimal}. We also describe the connections between our work and the work of Mirkin et al.~\cite{mirkin2021dead} in Appendix ~\ref{sec:discrete_optimal}.

\subsection*{Notation} We let $I_n$ and $0_n$ denote the $n\times n$
identity and zero matrices, respectively, with the subscript
suppressed if $n=1$. Also, we let $0_{n\times m}$ denote an
$n\times m$ matrix of zeros, and $I_{n\times m}$ denote a block matrix
whose upper left hand block is the identity matrix $I_{\min(m,n)}$ and
whose remaining entries are zero. Finally, let $e_i$ denote the $i$'th
standard basis vector.

 \section{Feedback of Estimated Plant States Partially Inverts Delay}
 \label{sec:continuous}
 In this section, we describe a standard, state-estimate
   feedback controller in which the observer includes a Pad\'{e}
   approximation of the delay.  This approach performs delay
   compensation, in the sense that the estimated plant states
   partially invert the delay. 

 Consider the single input, single output  linear system
 \begin{equation}
     \dot{x} = Ax + Bu+ Ed, \qquad y = Cx, \qquad x\in \mathbb{R}^n, \label{eq:xdot}
 \end{equation}
and denote the transfer function from $u$ to $y$ by $G(s) = C(sI_n-A)^{-1}B$ and that from $d$ to $y$ by $G_d(s)=C(sI_n-A)^{-1}E$. Assume that $(A,C)$ is observable and that $(A,B)$ and $(A, E)$ are controllable. 

Suppose that the measurement is delayed by $\tau$ seconds, so that only the delayed output
 \begin{equation}\label{eq:w(t)}
    w(t) = y(t-\tau)
 \end{equation}
  is available to the controller, and assume the presence of additive measurement noise
 \begin{equation}\label{eq:wm}
    w^m = w+n.
 \end{equation}
 
 To obtain a finite dimensional system, we will approximate the time delay $e^{-s\tau}$ by passing $y(t)$ through an $N$'th order Pad\'e approximation ~\cite{natori2012design,probst2010using,Franklin} with minimal  realization
  \begin{equation} \label{eq:qdot}
  \dot{q}_N = A_Nq_N + B_Ny, \;w_N = C_Nq_N+D_Ny, \;q_N\in \mathbb{R}^N       
 \end{equation}
 and transfer function 
   \begin{equation} \label{eq:PN}
 P_N(s)=C_N(sI_N-A_N)^{-1}B_N+D_N
 \end{equation}
 For example, with $N=1$,
$P_1(s) =(2-\tau s)/(2+\tau s)$.
 Note that $P_1(s)$ has a nonminimum phase zero at $z = 2/\tau$. It  is generally true that an $N$-th order Pad\'e approximation will have $N$ nonminimum phase zeros $\{z_i, i = 1, \ldots,N\}$, $N$ poles at the arithmetic inverse of these zeros, and is allpass with unity gain: $|P_N(j\omega)|=1, \forall \omega$.

Denote the system obtained by augmenting the Pad\'e state equations \eqref{eq:qdot} to those of the plant \eqref{eq:xdot} with noisy measurement \eqref{eq:wm}  by
 \begin{align}
     \dot{x}_{aug} &= A_{aug}x_{aug} + B_{aug}u +  E_{aug}d,  \label{eq:xaugdot}\\
              w^m_N &= C_{aug}x_{aug} +n,    \label{eq:wN1}
 \end{align}
 where $x_{aug} = \begin{bmatrix}x^T & q_N^T\end{bmatrix}^T$,
 \begin{equation}\label{eq:aug_defs}
 A_{aug} = \begin{bmatrix}A & 0_{n\times N}\\ B_NC & A_N\end{bmatrix}, \; B_{aug} = \begin{bmatrix}B \\ 0_{N\times 1}\end{bmatrix},
 \end{equation}	
 $E_{aug} = \begin{bmatrix}E^T & 0_{N\times 1}\end{bmatrix}^T$ and $C_{aug} = \begin{bmatrix}D_NC&C_N\end{bmatrix}$.
It is straightforward to show that if $G(s)$ has no zeros at the eigenvalues of $A_N$, then $(A_{aug}, B_{aug})$ is controllable. Similarly, if $P_N(s)$ has no zeros at the eigenvalues of $A$, then $(A_{aug}, C_{aug})$ is observable. 

 Let the control law be given by state estimate feedback
 \begin{equation}\label{eq:u}
    u = -K\hat{x} + Hr,
 \end{equation}
 where $\hat{x}$ must be obtained using the delayed measurement of the output $y$, which we have approximated by passing $y$ through the Pad\'e approximation $P_N(s)$. Hence an observer must estimate both the plant states and the states of \eqref{eq:qdot}. Denote the estimator gain for the augmented system by $L^T_{aug} = \begin{bmatrix}L_1^T& L_2^T\end{bmatrix}^T$, with $L_1\in \mathbb{R}^n$ and  $L_2\in \mathbb{R}^N$. Then
 \begin{align}
     \dot{\hat{x}}_{aug} &= A_{aug} \hat{x}_{aug} +  B_{aug}u +  L_{aug}\tilde{w}^m_N \label{eq:obs}\\
  \hat{w}_N &= C_{aug}\hat{x}_{aug}. \label{eq:what}
\end{align} 
where $\tilde{w}^m_N$, the measured estimation error for the delayed output, is given by
\begin{equation}\label{eq:wmtilde}
  \tilde{w}^m_N = w^m - \hat{w}_N.
\end{equation}
Substituting the control law $u = -K_{aug}\hat{x}_{aug}+Hr$, where $K_{aug} = \begin{bmatrix}K&0\end{bmatrix}$, and    $r$ is set to zero,  and using the fact that $\hat{w}_N$ satisfies \eqref{eq:what} yields that the transfer function from $w^m$ to $-u$ is given by $-U(s)=\bar{C}_{obs}(s)W^m(s)$, where $\bar{C}_{obs}(s)$ has the state variable realization 
\begin{equation}\label{eq:Cobs_ss}
 \dot{\hat{x}}_{aug}  =A_{obs}\hat{x}_{aug}+L_{aug}w^m,\quad -u = K_{aug}\hat{x}_{aug},
\end{equation}
and 
\begin{equation}\label{eq:Aobs}
A_{obs}=A_{aug}-B_{aug}K_{aug}-L_{aug}C_{aug}. 
\end{equation}
Under mild assumptions, the system \eqref{eq:Cobs_ss} is minimal (see Lemma~\ref{prop:Cobs_minimal}  in the Appendix).

We now study the transmission zeros of the observer based compensator $\bar{C}_{obs}(s)$. Given a state variable system $(A, B, C)$, recall that a zero of the Rosenbrock System Matrix will also be a transmission zero of the associated transfer function if $(A,B)$ is controllable and $(A,C)$ is observable \cite{AM97}.
The followeing lemma, whose proof is in the Appendix \ref{sec:continuous_outputdelay}, shows that  under very mild assumptions the realization \eqref{eq:Cobs_ss} is minimal.
 \begin{lemma}\label{prop:Cobs_minimal}Consider the  compensator defined by \eqref{eq:Cobs_ss}. 
 \begin{enumerate}[(i)]
 \item Assume that the eigenvalues of $A_N$ and $A-BK$ are disjoint,  and that $A-BK$ has no eigenvalues that are uncontrollable from $L_1$. Denote the eigenvalues and associated left eigenvectors of $A_N$ by $\lambda_i$ and $w_i^T$, $ i = 1, \ldots, N$, and assume further that
 \begin{equation}\label{eq:wicondition}
 w_i^TB_NC(\lambda_i I_n-A+BK)^{-1}L_1+w_i^TL_2 \neq 0.
 \end{equation}
  Then $(A_{obs}, L_{aug})$ is a controllable pair. 
\item Assume that the eigenvalues of $A_{aug}-L_{aug}C_{aug}$ and $A-L_1D_NC$ are disjoint, and that $K(\lambda I_n-A+L_1C)^{-1}L_1\neq 0$ for any $\lambda$ that is an eigenvalue of $A_{aug}-L_{aug}C_{aug}$. Then $(A_{obs}, K_{aug})$ is an observable pair.\qed
 \end{enumerate}
 \end{lemma}

\begin{proposition}\label{cor:Cobszeros}Consider the transfer function $\bar{C}_{obs}(s)$ with state variable realization \eqref{eq:Cobs_ss}, and assume that the hypotheses of Lemma~\ref{prop:Cobs_minimal} are satisfied.  Then the transmission zeros of $\bar{C}_{obs}(s)$ include zeros   at the open left half plane mirror images of the $N$ nonminimum phase zeros of $P_N(s)$.
\end{proposition}
\begin{proof}
Using definition \eqref{eq:qdot}  of  $A_N$, $B_N$, $C_N$, and $D_N$, the Rosenbrock System Matrix for this system is given by
\begin{align*} RS&M(s)=\\
&\begin{bmatrix}sI_n-A_{aug}+B_{aug}K_{aug}+L_{aug}C_{aug}& -L_{aug}\\
K_{aug} &0
  \end{bmatrix}.
\end{align*}
Properties of Pad\'e approximations imply that $A_N$ has $N$ distinct eigenvalues in the open left half plane located at the arithmetic inverse of the $N$ nonminimum phase zeros of $P_N(s)$. Let $\lambda_i=-z_i$ denote one such eigenvalue and let $v_i$ denote an associated eigenvector. Then $RSM(\lambda_i)$ has a nullspace spanned by    $\begin{bmatrix}0_{n\times 1}^T&v_i^T& (C_Nv_i)^T\end{bmatrix}^T$.
Since the hypotheses of Lemma~\ref{prop:Cobs_minimal} are satisfied, we know that the realization of $\bar{C}_{obs}(s)$ is minimal, and thus the  zeros of the Rosenbrock system matrix are also transmission zeros.
\end{proof}

A block diagram of the feedback system is in
Fig.~\ref{fig:continuous_system}.  Note that the zeros at the Pad\'e
poles in $\bar{C}_{obs}(s)$ provide phase lead to partially compensate
 for the phase lag due to the time delay. 
This is consistent with the results of Carver et
al.~\cite{Carver09}; specifically, the zeros in the compensator at the
poles of the Pad\'e approximation will partially invert the sensor
delay, and thus tend to mask the presence of the delay in the mapping
from $y$ to $u$.
\begin{figure}[htbp]
\centerline{\epsfig{file=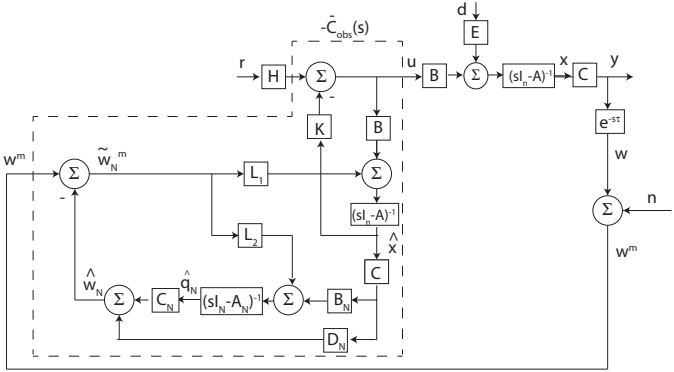,clip=,width = 3.5in}}
\caption{Continuous-time feedback system with  Pad\'e approximation to delay.}
\label{fig:continuous_system}
\end{figure}

Denote the estimation error between the delayed output $w$ and the estimate $\hat{w}_N$ by
\begin{equation}\label{eq:wtilde}
  \tilde{w}_N = w-\hat{w}_N.
\end{equation}
We now show that $\tilde{w}_N$ depends on the disturbance $d$, the
measurement noise $n$, and the error in the Pad\'e approximation to
the time delay.
\begin{proposition}\label{prop:F}
  Define the \emph{tradeoff filter}
  \begin{equation} \label{eq:Fs} \bar{F}(s) =
    \left(1+L_{est}^{aug}(s)\right)^{-1}.
  \end{equation}
  where $L_{est}^{aug}(s)$ denotes the transfer function from
  $\tilde{w}_N^m$ to $\hat{w}_N$ in
  Fig.~\ref{fig:continuous_system}, given by

  \begin{equation}\label{eq:Lestaug}
  \begin{aligned} 
      L_{est}^{aug}(s) &=  C_{aug}\left(sI_{n+N}-A_{aug}\right)^{-1}L_{aug}\\
      &=
        P_N(s)C\large(sI_{n}-A\large)^{-1}L_1 \\
    &\qquad\qquad + C_N\large(sI_{N}-A_N\large)^{-1}L_2
    \end{aligned}
  \end{equation}
  Then the estimation error \eqref{eq:wtilde} satisfies
  \begin{align}
    \tilde{W}_N(s) &= \bar{F}(s)\left( \Delta(s)G(s)U(s)+ e^{-s\tau} G_d(s)D(s)\right)\notag \\ & \qquad\qquad + \left(\bar{F}(s)-1\right)N(s),\label{eq:Fs_tradeoff}
  \end{align}
  where $\Delta(s) = e^{-s\tau}-P_N(s)$.
\end{proposition}
\begin{proof}
It follows from \eqref{eq:xdot}-\eqref{eq:w(t)}   that
\begin{equation}\label{eq:W(s)}
W(s)=e^{-s\tau}G(s)U(s) + e^{-s\tau} G_d(s)D(s).
\end{equation}
and follows from \eqref{eq:obs}-\eqref{eq:what} that 
   \begin{align}
  \hat{W}_N(s)         &= C_{aug}\left(sI_{n+N}-A_{aug}\right)^{-1}(B_{aug}U(s)+L_{aug}\tilde{W}^m_N)\notag \\ &= P_N(s)G(s)U(s) + L_{est}^{aug}(s)\tilde{W}^m_N \label{eq:What_s}
\end{align}
Together \eqref{eq:wm}, \eqref{eq:wmtilde}, \eqref{eq:W(s)} and \eqref{eq:What_s} imply that
\begin{equation} \label{eq:Wmtildes}
\tilde{W}_N^m(s)= \bar{F}(s)\left( \Delta(s)G(s)U(s)+ e^{-s\tau} G_d(s)D(s)+N(s)\right),    
\end{equation}
where $\bar{F}(s)$ is given by \eqref{eq:Fs}. Noting that $\tilde{w}^m_N = \tilde{w}_N+n$ together with \eqref{eq:Wmtildes} yields \eqref{eq:Fs_tradeoff}.


\end{proof}

Proposition~\ref{prop:F} implies that the filter $\bar{F}(s)$ describes a \emph{tradeoff} between the response of the estimation error to the plant disturbance $D(s)$ and the measurement noise $N(s)$. In addition, $\bar{F}(s)$ describes the response of the estimation error to $\Delta(s)$, the difference between the time delay and the Pad\'e approximation. We note that the first term in \eqref{eq:Fs_tradeoff}, which is proportional to $\Delta(s)$, will also depend on $D(s)$ and $N(s)$ through the control input $U(s)$. However, this first term can be minimized by increasing the degree of the Pad\'e approximation without affecting  the tradeoff between $D(s)$ and $N(s)$. Using large estimator gains $L_1$ and/or $L_2$ will force $\bar{F}(s)\approx 0$ and the disturbance response will be small at the expense of the noise being passed directly to the estimation error. The same will be true for the Pad\'e approximation error, provided that it is not large enough to destabilize the system. Using small estimator gains has the opposite effect. 

%

Let us now calculate the response of the system output. 
\begin{corollary}\label{cor:Yresponse}
The response of $y$ to the reference input $r$, disturbance $d$, and measurement noise $n$ is given by
\begin{align*}
   Y(s)&= C(sI_n-A+BK)^{-1}BHR(s)+ G_d(s)D(s)\\
         & - C(sI_n-A)^{-1}BK(sI_n-A+BK)^{-1}L_1\tilde{W}_N^m(s),
\end{align*}
where
\begin{equation*} 
    \tilde{W}_N^m(s) =    \bar{F}(s)\left( \Delta(s)G(s)U(s)+ e^{-s\tau}G_d(s)D(s)+N(s)\right).
\end{equation*}
\end{corollary}
\begin{proof}
The definition of $u$ and $\hat{x}$ yield 
\begin{align}\notag
   U(s) &= -K(sI_n-A+BK)^{-1}L_1\tilde{W}_N^m(s) \\&+ \left(1+K(sI_n-A)^{-1}B\right)^{-1}HR(s),\label{eq:U(s)}
\end{align}
and substituting \eqref{eq:U(s)} into \eqref{eq:xdot} yields the result.
\end{proof}

It follows from Corollary~\ref{cor:Yresponse} that,  in the absence of disturbances and measurement noise, the response of $y$ to a command $r$ is the same with the observer as with state feedback, except for the discrepancy caused by the error in the Pad\'e approximation to the time delay.  Specifically, in the absence of disturbances and noise,  the command response is given by
\begin{align}\notag
Y(s) &= G(s)\left(1+\bar{\Delta}(s)\right)^{-1}\\ &\times\left(1+K(sI_n-A)^{-1}B\right)^{-1}HR(s),\label{eq:YandDelta}
\end{align}
where $\bar{\Delta}(s)=K(sI_n-A+BK)^{-1}L_1\bar{F}(s)\Delta(s)G(s)$.
Hence we see that \emph{the bandwidth of the state feedback loop should be limited to the frequency range for which $\Delta(j\omega)\approx 0$}. Higher order  Pad\'e approximations can approximate the delay over a wider frequency range at the expense of the additional zeros in $\bar{C}_{obs}(s)$  amplifying sensor noise.



 \section{Decomposition of Optimal Estimator with Delay}
 \label{sec:optimal}
 We now describe special decomposition properties of the estimation problem from Section~\ref{sec:continuous} that arise when the estimator is \emph{optimal}.
To formulate this problem, we consider the augmented system
 \eqref{eq:xaugdot}-\eqref{eq:aug_defs}, suppose that $d$ and $n$ are
 zero mean Gaussian white noise processes with covariances $V\geq 0$
 and $W>0$, respectively, and   assume that $(A_{aug},C_{aug})$ is
 observable and   $(A_{aug},E_{aug})$ is controllable. Then the
 optimal estimator gains satisfy
\begin{equation}\label{eq:Laugopt}
    L_{aug} = \Sigma_{aug} C_{aug}^T W^{-1},
\end{equation}
where $\Sigma_{aug}$  is the unique positive definite solution to the algebraic Riccati equation
\begin{align}\notag
 0 = A_{aug}\Sigma_{aug} &+ \Sigma_{aug} A_{aug}^T +V_{aug}\\&- \Sigma_{aug} C_{aug}^TW^{-1}C_{aug}\Sigma_{aug},\label{eq:Ric_aug}
\end{align}
with $V_{aug} = E_{aug}VE_{aug}^T$. 
 
 By substituting \eqref{eq:what} into \eqref{eq:obs} and applying \eqref{eq:aug_defs}, we see that the optimal estimator has state variable description
 \begin{align}\label{eq:xaughat}
 \dot{\hat{x}}_{aug} &= A_{aug}^{CL}\hat{x}_{aug}+L_{aug}w^m + B_{aug}u\\
 \hat{w}_N &= C_{aug}\hat{x}_{aug},\label{eq:wNhat}
 \end{align}
 where $A_{aug}^{CL} = A_{aug}-L_{aug}C_{aug}$ satisfies
 \begin{equation}\label{eq:ACLaug}
    A_{aug}^{CL}  =   \begin{bmatrix}A-L_1D_NC & -L_1C_N\\B_NC-L_2D_NC&A_N-L_2C_N\end{bmatrix}.
 \end{equation}
  
  We now characterize the $N+n$ eigenvalues of \eqref{eq:ACLaug} when $L_{aug}$ is the optimal estimator gain given by \eqref{eq:Laugopt}. 
   
 \begin{proposition}\label{prop:Ham} Consider the optimal estimator defined by \eqref{eq:Laugopt}-\eqref{eq:ACLaug}. Define   $L = \Sigma C^TW^{-1}$, where  $\Sigma$ is the unique positive  definite solution to the Riccati equation
 \begin{equation}\label{eq:Ric_lower}
     0 = A\Sigma + \Sigma A^T-\Sigma C^TW^{-1}C\Sigma +EVE^T.
 \end{equation}
 Assume that the eigenvalues of $A_N$, $A$,  and $A-LC$ are disjoint. Then $A_{aug}^{CL}$ has
 \begin{enumerate}[(i)]
 \item $N$ eigenvalues identical to those of $A_N$, and
 \item $n$ eigenvalues identical to those  of $A-LC$. 
\end{enumerate}

 \end{proposition}
 
 \begin{proof} 
 It is well known that the eigenvalues of $A_{aug}^{CL}$ are equal to the open left half plane  eigenvalues of the  Hamiltonian matrix associated with the Riccati equation \eqref{eq:Ric_aug}, given by
\begin{equation}\label{eq:Hamaug}
H_{aug} =  \begin{bmatrix}A_{aug}^T & -C_{aug}^TW^{-1}C_{aug}\\
    -E_{aug}VE_{aug}^T & -A_{aug}\end{bmatrix}.
\end{equation}
Substituting \eqref{eq:aug_defs} into \eqref{eq:Hamaug} yields
\begin{equation}\label{eq:Hamblock}
  \scriptsize{
  H_{aug} = \begin{bmatrix}A^T & C^TB_N^T & -C^TD_N^TW^{-1}D_NC & -C^TD_N^TW^{-1}C_N\\0 &A_N^T & -C_N^TW^{-1}D_NC&-C_N^TW^{-1}C_N\\
    -EVE^T & 0 & -A & 0\\
    0&0&-B_NC&-A_N\end{bmatrix}
    }
\end{equation}
Let $\lambda$ be \emph{any} eigenvalue of $H_{aug}$ and let $\nu$ denote an associated right eigenvector:  $H_{aug}\nu=\lambda \nu$. Partition $\nu$ as $\nu = \begin{bmatrix}w^T&v^T&x^T&y^T\end{bmatrix}^T$, where $w,x\in\mathbb{R}^n$ and $v,y\in\mathbb{R}^N$. Then \eqref{eq:Hamblock} implies that 
\begin{align}
A^Tw+ C^TB_N^Tv + C^TD_N^TW^{-1}(D_NCx+C_Ny)&= \lambda w ,\label{eq:1}\\
A^T_Nv +  C_N^TW^{-1}(D_NCx + C_Ny) &= \lambda v, \label{eq:2}\\
-EVE^Tw - Ax &= \lambda x,\label{eq:3}\\
 -B_NCx - A_Ny&= \lambda y.\label{eq:4}
\end{align}
It follows directly from \eqref{eq:3}-\eqref{eq:4} that $x$ and $y$ must satisfy
 \begin{align}
      x &= -(\lambda I_n+A)^{-1}EVE^Tw,  \label{eq:xi}  \\
      y &=  -(\lambda I_N+A_N)^{-1}B_NCx,\label{eq:yi}
 \end{align}
 and invoking \eqref{eq:PN} yields the additional constraint
 \begin{equation}
    D_NCx + C_Ny = P_N(-\lambda)Cx.
 \end{equation}
We now treat the two sets of eigenvalues separately.
\begin{enumerate}[(i)]
\item Let $\lambda$ be an eigenvalue of $A_N$ and let $v\in\mathbb{C}^N$ denote an associated right eigenvector of $A_N^T$:
\begin{equation}\label{eq:vi}
A_N^Tv=\lambda v. 
\end{equation}
The fact that the Pad\'e approximation $P_N(s)$ has right half plane zeros at the mirror images of the eigenvalues of $A_N$ implies that $P_N(-\lambda_i)=0$. Hence \eqref{eq:2} is satisfied by definition, and \eqref{eq:1} will hold provided that
 \begin{equation}\label{eq:wi}
      w = (\lambda I_n-A^T)^{-1}C^TB_N^Tv. 
 \end{equation}

\item Denote the Hamiltonian matrix associated with the Riccati equation \eqref{eq:Ric_lower} by 
 \begin{equation}\label{eq:Hblock}
  H = \begin{bmatrix}A^T&-C^TD_N^TW^{-1}D_NC\\-EVE^T & -A\end{bmatrix}.
\end{equation} 
Let   $\lambda$ be an open left half plane eigenvalue of $H$ with right eigenvector  $\begin{bmatrix}w^T&x^T\end{bmatrix}^T$. Then $w$ and $x$ must satisfy \eqref{eq:3} and 
\begin{equation}\label{eq:1a}
A^Tw-C^TD_N^TW^{-1}D_NCx=\lambda w. 
\end{equation}
Define $y$ as in \eqref{eq:yi} and 
\begin{equation}\label{eq:vdef}
    v = (\lambda I_N-A_N)^{-1}C_N^TW^{-1}P_N(-\lambda)Cx.
\end{equation}    
Then $\nu = \begin{bmatrix}w^T&v^T&x^T&y^T\end{bmatrix}^T$ satisfies \eqref{eq:2}-\eqref{eq:4}. To show that $\nu$ is an eigenvector of $H_{aug}$ with eigenvalue $\lambda$ it remains to prove that \eqref{eq:1} holds. Substituting \eqref{eq:1a} and applying \eqref{eq:yi} and \eqref{eq:vdef} shows that \eqref{eq:1} reduces to 
\begin{align}\notag
    &C^T B_N^Tv -C^TD_N^TW^{-1}C_Ny\\ &= C^TB_N^T(\lambda I_N-A_N^T)^{-1}C_N^TW^{-1}P_N(-\lambda)Cx\notag\\
     &+ C^TD_N^TW^{-1}C_N(\lambda I_N+A_N)^{-1}B_NCx.\label{eq:1x}
\end{align}
It follows from \eqref{eq:PN} that $C_N(\lambda I_N\pm A_N)^{-1}B_N=D_N-P_N(\mp \lambda)$, and thus that
\begin{align}\notag 
&C^T B_N^Tv- C^TD_N^TW^{-1}C_Ny  \\
&\quad =  (D_N-P_N(\lambda))W^{-1}P_N(-\lambda)\notag\\ &\quad +D_N^TW^{-1}(D_N-P_N(-\lambda))\label{eq:1y}
\end{align}
The fact that $P_N(s)$ is a Pad\'e approximation implies that $P_N(s)P_N(-s)=1$ and thus that $D_N = \pm 1$. Together these facts show that \eqref{eq:1y} reduces to zero, and thus that $\nu$ satisfies \eqref{eq:1} and is an eigenvector of $H_{aug}$.

\end{enumerate}
\end{proof}

Proposition~\ref{prop:Ham} implies that  if the
estimator is optimal  then the value of the gain $L$ does not depend
on the delay in the feedback measurement. Essential to the proof are the facts that the process noise $d$ directly affects only the plant states upstream of the delay and that the delay and its Pad\'e   approximation are allpass.

We next show that the eigenvalues of $A_N$ that are shared between $A_{aug}$ and $A_{aug}^{CL}$ are uncontrollable from the estimator gain $L_{aug}$. To do so, we first present a preliminary result showing that if certain eigenvalues of a system are preserved under state feedback, then these eigenvalues are unobservable in the state feedback.  
 \begin{lemma}\label{lem:FGK}Consider the linear system $\dot{x}=Fx+Gu$, where $x\in\mathbb{R}^n$, $u\in \mathbb{R}$, together with   the state feedback $u=-Kx$.  Assume that and $(F,G)$ is controllable and that $F$ and $F-GK$ have $m$ eigenvalues in common. Then these eigenvalues are unobservable eigenvalues of $(F,K)$. 
\end{lemma}

\begin{proof} 
 Assume that $\lambda$ is a common eigenvalue of $F$ and $F-GK$. The latter fact implies there exists a nonzero $v_1^{\prime}$ such that $(\lambda I-F+GK)v^{\prime}_1=0$;  equivalently
\begin{equation}\label{eq:v1prime}
  \begin{bmatrix}
     \lambda I - F & G
  \end{bmatrix}\begin{bmatrix}
     v_1^{\prime}\\
     Kv_1^{\prime}
  \end{bmatrix} =0.
\end{equation}
Since $\lambda$ is also an eigenvalue of $F$, $\rank\left(\lambda I -F\right) < n$, and $(F, G)$ controllable implies that $\rank\begin{bmatrix}\lambda I -F&G\end{bmatrix} = n$. It follows that  $G$ cannot be a linear combination of the columns of $\lambda I - F$, 
and thus the solution $v_1^{\prime}$ to \eqref{eq:v1prime} must satisfy
 \begin{equation}\label{eq:FGK_F_eval_v1pri}
\left(\lambda I - F\right)v_1^{\prime} = 0
\end{equation}
 \begin{equation}\label{eq:K_v1pri}
Kv_1^{\prime} = 0.
\end{equation}
Together, \eqref{eq:FGK_F_eval_v1pri}-\eqref{eq:K_v1pri} imply that  
\begin{equation}
  \rank\begin{bmatrix}
     \lambda I - F\\
     K
  \end{bmatrix} < n,
\end{equation}
and hence $\lambda$ is an unobservable eigenvalue of $\left(F,K\right)$. 

Next suppose that $\lambda$ is a common eigenvalue of $F$ and $F-GK$ with algebraic multiplicity $M>1$. Then controllability implies that the geometric multiplicity of $\lambda$ is equal to one in both cases. We  now show that $\lambda$ is an unobservable eigenvalue of $(F, K)$ with multiplicity $M$. Specifically, if  $(\bar{F}, \bar{K})$  denotes the Jordan canonical form of $(F, K)$, then $\bar{F}$ will have a single $M$-dimensional Jordan block associated with $\lambda$ for which all corresponding entries of $\bar{K}$ are equal to zero.

First, there exists a chain  of generalized eigenvectors $v_1,v_2,...,v_M$  with 
  \begin{align}
     \left(\lambda I - F\right)v_1 &= 0,  \label{eq:F_eval1}\\
     \left(\lambda I - F\right)v_{k+1} &= v_k,\;k=1,\ldots,M-1. \label{eq:F_evalk}
     \end{align}
     A similarity transformation that places the system in Jordan canonical form is given by $Q=\begin{bmatrix}v_1&\cdots&v_M&v_{M+1}&\cdots&v_n\end{bmatrix}$, where $v_{M+1} \ldots v_n$ represent the eigenstructure of the remaining eigenvalues. Then $\bar{F}=QFQ^{-1}$ and $\bar{K}=KQ$, and the result will be proven if we can show that $Kv_i=0,\; i=1,\ldots M$.
     
There also exists a chain of generalized eigenvectors $v_1^\prime,v_2^\prime,...,v_M^\prime$ with
     \begin{align}
       \left(\lambda I - F+GK\right)v_1^\prime &= 0,  \label{eq:F_eval1p}\\
     \left(\lambda I - F+GK\right)v_{k+1}^\prime &= v_k^\prime,\;k=1,\ldots,M-1. \label{eq:F_evalkp}  
 \end{align}
Since the geometric multiplicity of $\lambda$ is equal to one, it follows from \eqref{eq:FGK_F_eval_v1pri} and \eqref{eq:F_eval1} that 
\begin{equation}\label{eq:v1}
   v_1^{\prime} = \alpha_1 v_1
 \end{equation}
for some nonzero $\alpha_1$. Hence \eqref{eq:K_v1pri} and \eqref{eq:v1} imply that 
 \begin{equation}\label{eq:K_v1}
 Kv_1 = 0.
\end{equation}




It remains to show that $Kv_i=0, \;i=2,\ldots, M$. We shall also need to generalize the intermediate result \eqref{eq:v1} to apply to the remaining generalized eigenvectors $v_i^\prime$ associated with $\lambda_i$. 
Hence  we assume  there exists $\ell<M$ and nonzero scalars $\alpha_1, \ldots, \alpha_\ell$ such that, for $i=1,\ldots,\ell$,
\begin{align}
  v_i^\prime &=\alpha_1v_i+\ldots+\alpha_i v_1\label{eq:viprime}\\
    Kv_i&=0,\label{eq:Kvi0}
\end{align}
and show    that \eqref{eq:viprime}-\eqref{eq:Kvi0} must also hold for $i=\ell+1$. To do so, we first apply \eqref{eq:F_evalkp} and \eqref{eq:viprime} with   $k=\ell$,  and \eqref{eq:F_evalk} with $k=1,\ldots,\ell$, yielding
\begin{align}
(\lambda I-F+GK)v_{\ell+1}^\prime &= v_\ell^\prime\notag\\
&= \alpha_1v_\ell+\ldots+\alpha_\ell v_1\notag\\
&= (\lambda I-F)(\alpha_1 v_{\ell+1}+\ldots+\alpha_l v_2).\label{eq:tempeqn1}
\end{align}
Rearranging \eqref{eq:tempeqn1} results in
\begin{equation}\label{eq:FGK_evalkplus1_matrix_kplus1}
  \begin{bmatrix}
     \lambda I - F & G
  \end{bmatrix}\begin{bmatrix}
     v_{\ell+1}^{\prime} - \alpha_1 v_{\ell+1}- ... - \alpha_{\ell} v_2\\
     Kv_{\ell+1}^{\prime}
  \end{bmatrix} = 0,
\end{equation}
thus implying
 \begin{equation}\label{eq:FGK_F_eval_vkplus1pri}
\left(\lambda I - F\right)\left( v_{\ell+1}^{\prime} - \alpha_1 v_{k+1}- ... - \alpha_{\ell} v_2 \right)= 0
\end{equation}
 \begin{equation}\label{eq:K_vkplus1_pri}
Kv_{\ell+1}^{\prime} = 0.
\end{equation}
 The fact that $\lambda$ has   geometric multiplicity  equal to one, together with  \eqref{eq:FGK_F_eval_v1pri} and    \eqref{eq:FGK_F_eval_vkplus1pri}, implies that  $v_{\ell+1}^{\prime} - \alpha_1 v_{\ell+1}- \ldots - \alpha_{\ell} v_2 = \alpha_{\ell+1} v_1$, for some where $\alpha_{\ell+1}\neq 0$, and thus
 \begin{equation}
     v_{\ell+1}^{\prime} = \alpha_1 v_{\ell+1}+ \ldots + \alpha_{\ell} v_2 + \alpha_{\ell+1} v_1.
\end{equation}
 It then follows from \eqref{eq:Kvi0} and  \eqref{eq:K_vkplus1_pri}  that 
 \begin{equation}\label{eq:K_vkp1}
Kv_{\ell+1} = 0,
\end{equation}
thus completing the induction. 
\end{proof}

 \begin{proposition}\label{prop:unctrb2}
 The $N$ eigenvalues of $A_{aug}^{CL}$ that are identical to those of $A_N$ are uncontrollable eigenvalues of $(A_{aug},L_{aug})$. Furthermore, a minimal realization of $L_{est}^{aug}(s)$ defined in \eqref{eq:Lestaug} is given by $(A, L, C)$ with transfer function
  \begin{equation}\label{eq:Lest}
     L_{est}(s) = C(sI_n-A)^{-1}L,
  \end{equation}
  with $L$ given by Proposition~\ref{prop:Ham}. 
 \end{proposition}
 \begin{proof}
The dual of Lemma~\ref{lem:FGK} states that if $(F,H)$ is observable and $F$ and $F-LH$ share $m$ eigenvalues, then these eigenvalues are uncontrollable eigenvalues of $(F,L)$.   Consider the estimation problem for the augmented system $(A_{aug}, E_{aug}, C_{aug})$ and denote the open loop transfer function from disturbance $d$ to estimation error $w_N^M$ by $G_d^{aug}(s)=C_{aug}\left(sI-A_{aug}\right)^{-1}E_{aug}$. It is straightforward to show that $G_d^{aug}(s)=P_N(s)G_d(s)$, where $P_N(s)P_N(-s)=1$. Hence the return difference equality \cite{KS72}  associated with the Riccati equation \eqref{eq:Ric_aug} reduces to
\begin{equation}\label{eq:returndiffaug}
W+G_d(s)VG_d(-s)=\left(1+L_{est}^{aug}(s)\right)W\left(1+L_{est}^{aug}(-s)\right).
\end{equation}
The left hand side of \eqref{eq:returndiffaug} must satisfy the return difference equality for the lower order Riccati equation \eqref{eq:Ric_lower} 
\begin{align}
    W+G_d(s)VG_d(-s)&=(1+C(sI-A)^{-1}L)W \notag\\
    &\qquad\times(1+C(-sI-A)^{-1}L),
\end{align}
and thus $L_{est}^{aug}(s)$ reduces to \eqref{eq:Lest}.

 \end{proof}


We saw in Proposition~\ref{prop:F} that the optimal estimation error is governed by the tradeoff filter $\bar{F}(s)$ defined in \eqref{eq:Fs}. With optimal estimation,  Proposition~\ref{prop:unctrb2} implies that $\bar{F}(s)$ reduces to
 \begin{align}\label{eq:Fsopt}
     F(s) &= \left(1+L_{est}(s)\right)^{-1}\notag\\
     &= \left(1+C(sI_n-A)^{-1}L\right)^{-1},
 \end{align}
 and thus is independent of the time delay. The optimal estimation error thus has the form
  \begin{equation}\label{eq:Fs_tradeoff2}
   \begin{aligned}
     \tilde{W}_N(s)  = {}& \large[F(s) G(s)U(s)\large] \cdot  \Delta(s) && \text{(Pad\'e error)} \\
     &+ \large[ F(s) e^{-s\tau} G_d(s)\large] \cdot D(s) && \text{(disturbance)} \\
     & + \large[F(s)-1\large] \cdot N(s), &&\text{(noise)}
   \end{aligned}
 \end{equation}
and depends on the time delay only through the Pad\'e approximation error $\Delta(s)$.

 \section{An Alternate Control Architecture} 
 \label{sec:alternate}
We now exploit the result of Proposition~\ref{prop:unctrb2} by proposing an alternate control architecture for the compensator that has several properties  remarkably similar to those of the Smith Predictor
   \cite{Smith57}.

 The state equations \eqref{eq:obs}-\eqref{eq:what} imply that the estimator for $\hat{w}_N$ in Fig.~\ref{fig:continuous_system} satisfies
 \begin{equation}     \label{eq:WNhat(s)} 
     \hat{W}_N(s)     = C_{aug}(sI_{n+N}-A_{aug})^{-1}L_{aug}\tilde{W}^m_N(s) + V_N(s), 
 \end{equation}
 where $V_N(s)=P_N(s)G(s)U(s)$.   Suppose, as shown in Fig.~\ref{fig:F}, that we implement this estimator using separate subsystems for the responses to the measured estimation error $\tilde{w}_N^m$ and to the input $u$. Doing so may seem counterintuitive as it leads to a compensator with higher dynamical order than that   depicted in Fig.~\ref{fig:continuous_system}. Indeed, the additional dynamics may imply that the compensator is no longer stabilizing. Nevertheless, we shall show that if the estimator is \emph{optimal}, then this control architecture has several interesting features, and we shall state conditions under which it does stabilize the plant. Then in Section~\ref{sec:Smith} we describe similarities between the architecture of Fig.~\ref{fig:F} and the Smith Predictor.
 
\begin{figure}[htbp]
\centerline{\epsfig{file=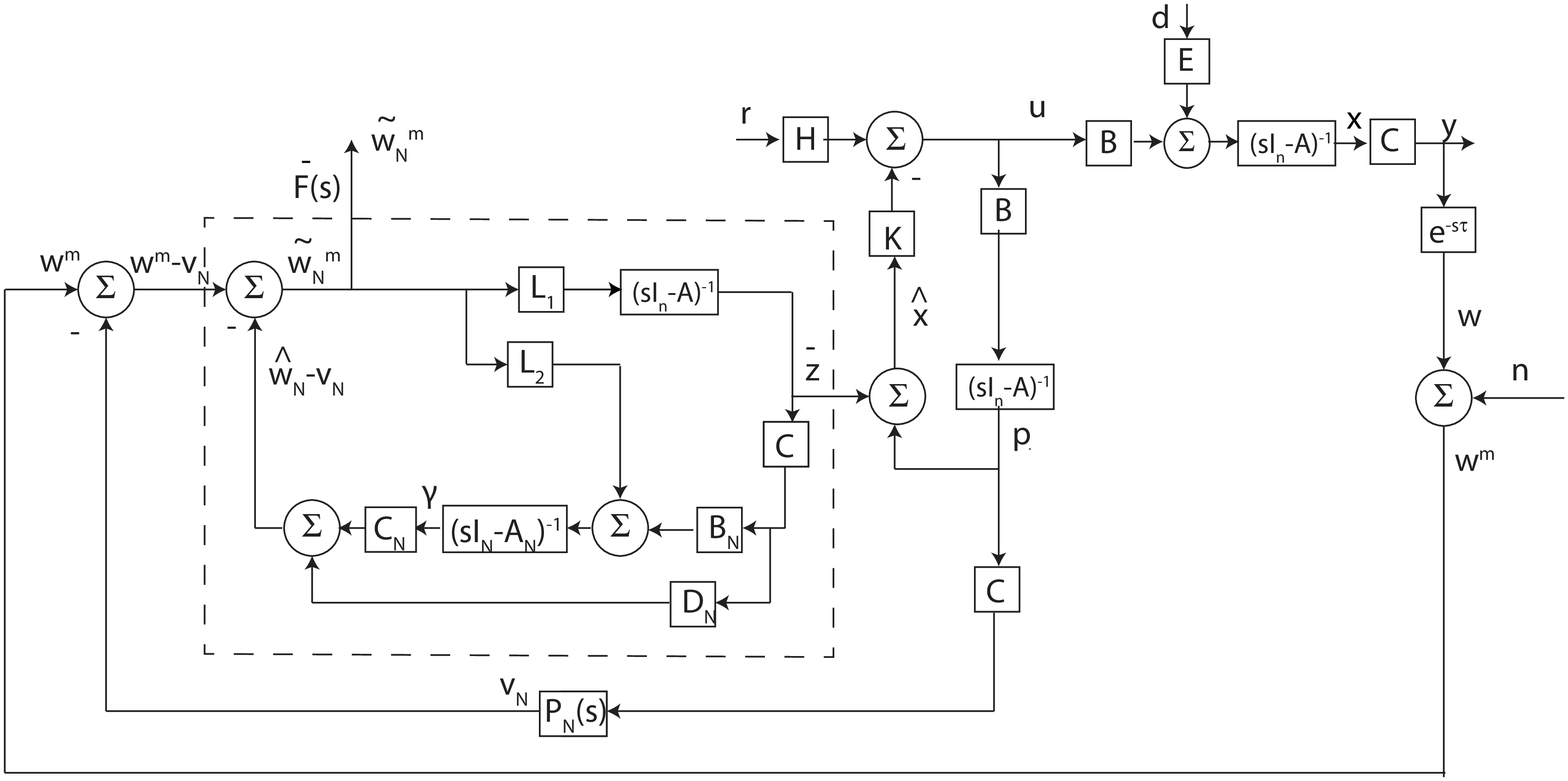,clip=,width=3.5in}}
\caption{Rearranged block diagram depicting the tradeoff filter $\bar{F}(s)$.}
\label{fig:F}
\end{figure}

 We have seen in Section~\ref{sec:optimal} that if the estimator gain is optimal, then \eqref{eq:WNhat(s)} reduces to 
  \begin{equation}\label{eq:WNhat(s)_opt}
     \hat{W}_N(s) = C(sI_n-A)^{-1}L\tilde{W}^m_N(s) + V_N(s).
 \end{equation}
A derivation similar to that used to prove Proposition~\ref{prop:F} shows that the tradeoff  \eqref{eq:Fs_tradeoff} between response to noise and response to disturbances and Pad\'e approximation error continues to hold, with $\bar{F}(s)$  defined by  \eqref{eq:Fs}  replaced by \eqref{eq:Fsopt}
 \begin{equation}\notag
     F(s) = \left(1+L_{est}(s)\right)^{-1}.
 \end{equation}
Using this fact, the block diagram in Fig.~\ref{fig:F} simplifies to that in Fig.~\ref{fig:F_optimal}, whose properties we now explore. 
 
To begin, we note that the observer based compensator $C_{obs}(s)$ has a  state variable description
\begin{equation}\label{eq:xCdot}
\dot{x}_C = A_Cx_C+B_Cw^m, \qquad u = -C_Cx_C,
\end{equation}
where $ x_C = \begin{bmatrix}z^T&p^T&\beta^T \end{bmatrix}^T$,
\begin{align}\label{eq:AC}
A_C &= \begin{bmatrix}A-LC&-LD_NC&-LC_N\\
   -BK&A-BK&0_{n\times N}\\
   0_{N\times n}&B_NC&A_N\end{bmatrix}, \; B_C=\begin{bmatrix}L\\0_{n\times 1}\\0_{N\times 1}\end{bmatrix},    
\end{align}
and $C_C=\begin{bmatrix}K&K&0_{1\times N}\end{bmatrix}$. 
The formula for the inverse of a block matrix and some algebra shows that
\begin{align}\label{eq:Cobs_re}
 C&_{obs}(s) =C_C(sI_{2n+N}-A_C)^{-1}B_C  \\&= \frac{K(sI_n-A+BK+LC)^{-1}L}{1+G(s)K(sI_n-A+BK+LC)^{-1}L(1-P_N(s))}.\notag
\end{align}

\begin{figure}[htbp]
\centerline{\epsfig{file=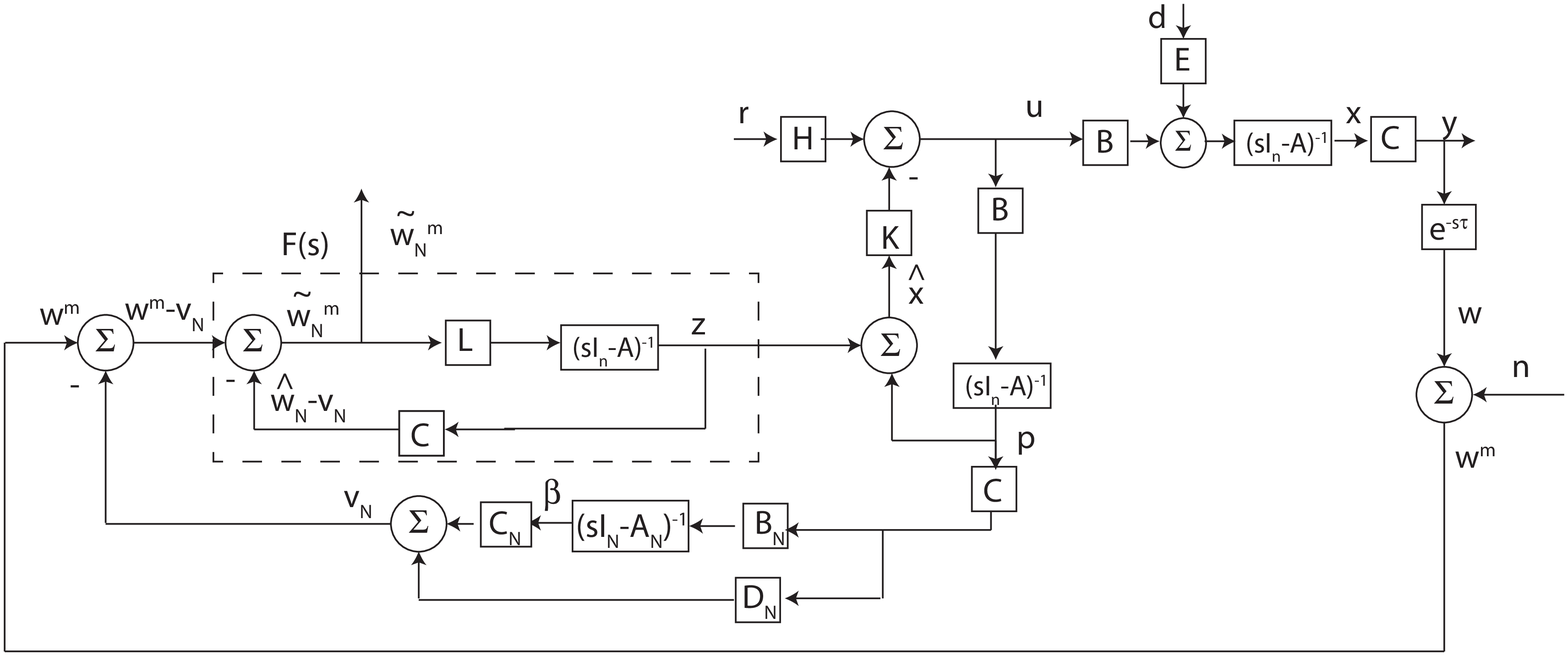,clip=,width=3.5in}}
\caption{Rearranged block diagram with an optimal estimator.}
\label{fig:F_optimal}
\end{figure} 
 
Unlike the original architecture in Fig.~\ref{fig:continuous_system}, the estimator gain $L$ may be chosen without regard for the delay, and has dimension determined only by that of the plant.  On the other hand, the compensator $C_{obs}(s)$ now has dynamical order $2n+N$, which is larger than that of the compensator in Fig.~\ref{fig:continuous_system}. Finally, note that the estimator gain $L$ in Fig.~\ref{fig:F_optimal} will generally differ from the gain $L_1$ in Fig.~\ref{fig:F}, and thus the former estimator cannot be obtained from the latter simply by setting $L_2=0$.

The next lemma, whose proof is in Appendix~\ref{sec:continuous_outputdelay}, shows that under mild assumptions the realization \eqref{eq:xCdot}-\eqref{eq:AC} will be minimal, with the exception that if   $\lambda = 0$ is an eigenvalue of $A$, then zero will be an  unobservable eigenvalue of $(A_C, C_C)$. 
 \begin{lemma}\label{prop:Cobsnewmin}Consider the compensator defined by \eqref{eq:xCdot}-\eqref{eq:AC}. 
 \begin{enumerate}[(i)]
 \item Assume that $A_N$ has no eigenvalues that are also zeros of $G(s)$, that $A_N$ and $A-BK$ have no eigenvalues that are zeros of $K(sI_n-A)^{-1}L$ and that the eigenvalues of $A_N$, $A$, and $A-BK$ are mutually disjoint. Then $(A_C, B_C)$ is a controllable pair.
 \item\label{lem:ptii} Assume that the eigenvalues of $A$, $A_N$, and $A-LC$ are mutually disjoint, and that  $(A-LC,K)$ is an observable pair.  Then if $A$ has no eigenvalues at the origin, the pair $(A_C, C_C)$ is  observable. If $A$ has an eigenvalue equal to zero, then $\lambda = 0$ is an  unobservable eigenvalue of $(A_C, C_C)$.
 \end{enumerate}
 \end{lemma}

  The following result shows that, as in Proposition~\ref{cor:Cobszeros}, the eigenvalues of $A_N$ will appear as zeros of $C_{obs}(s)$. Furthermore, with one important exception, $C_{obs}(s)$ will also have zeros at the eigenvalues of $A$. 
\begin{proposition}\label{prop:Cobs_pz}Consider $C_{obs}(s)$ defined by \eqref{eq:xCdot}-\eqref{eq:AC}, and assume that the hypotheses of Lemma~\ref{prop:Cobsnewmin} hold.  
\begin{enumerate}[(i)]
\item\label{pt1} Let $\lambda$ be an eigenvalue of $A_N$. Then $\lambda$ is a transmission zero of $C_{obs}(s)$. 
\item\label{pt2} Assume that $\lambda$ is a \emph{nonzero} eigenvalue of $A$. Then $\lambda$ is a transmission zero of $C_{obs}(s)$.
\item\label{pt3} Assume that $\lambda = 0$ is an eigenvalue of $A$ with multiplicity $m\geq 1$. Then $\lambda$ is  a transmission zero of $C_{obs}(s)$ with multiplicity $m-1$.
\end{enumerate}
\end{proposition}
\begin{proof} Consider $C_{obs}(s)$ defined in \eqref{eq:Cobs_re},  
introduce coprime factorizations $K(sI_n-A+BK+LC)^{-1}L=N_C(s)/D_C(s)$, $G(s)=N_G(s)/D_G(s)$, and $P_N(s)=N_P(s)/D_P(s)$, and note that $D_G(s)$ may be factored as $D_G(s)=s^m\bar{D}_G(s)$, where $\bar{D}_G(0)\neq 0$. Further note that 
\begin{equation}\notag
\frac{N_P(s)}{D_P(s)} = \frac{\prod_{i=1}^N(z_i-s)}{\prod_{i=1}^N(z_i+s)},
\end{equation}
from which it follows that
\begin{equation}\notag
D_P(s)-N_P(s) = 2s\sum_{i=1}^N\prod_{j\neq i}z_j +  higher\; order\; terms \;in  \;   s.
\end{equation}
Hence we may factor $D_P(s)-N_P(s)=s\Delta(s)$, where $\Delta(0)\neq 0$.

Substituting these factorizations into \eqref{eq:Cobs_re} and rearranging results in
\begin{equation}\label{eq:Cobs_cp}
C_{obs}(s) = \frac{N_C(s)\bar{D}_G(s)D_P(s)}{\bar{D}_G(s)D_C(s)D_P(s)+N_G(s)N_C(s)\Delta(s)}.
\end{equation}
The assumption that $C_{obs}(s)$ is stabilizing implies that $N_C(0)\neq 0$, and the fact that Pad\'e approximations are stable implies that $D_P(0)\neq 0$. 
Hence if $m=1$ then $C_{obs}(0)\neq 0$; otherwise, $C_{obs}(s)$ will have a zero at the origin of order $m-1$.
\end{proof}

We now analyze stability of the feedback system in Fig.~\ref{fig:F_optimal}  with the time delay replaced by a Pad\'e approximation, so that the system has a state variable description
\begin{align}\label{eq:xfb}
\dot{x}_{fb} = A_{fb}x_{fb} + B_{fb}\begin{bmatrix}r\\d\\n\end{bmatrix},\quad
       \begin{bmatrix}u\\y\end{bmatrix} = C_{fb}x_{fb}+\begin{bmatrix}H\\0\end{bmatrix}r,
\end{align}
where $x_{fb} = \begin{bmatrix}x^T&q_N^T&z^T&p^T&\beta^T \end{bmatrix}^T$, and
\begin{align}
    A_{fb} &= \begin{bmatrix}A_{aug} & -B_{aug}C_C\\B_CC_{aug} & A_C\end{bmatrix}, \label{eq:Afb}  \\  B_{fb} &= \begin{bmatrix}BH&E&0_{n\times 1}\\0_{N\times 1}&0_{N\times 1}&0_{N\times 1}\\0_{n\times 1}&0_{n\times 1}&L\\BH&0_{n\times 1}&0_{n\times 1}\\0_{N\times 1}&0_{N\times 1}&0_{N\times 1}\end{bmatrix}, \label{eq:Bfb}  \\
  C_{fb} &= \begin{bmatrix}0_{1\times n} & 0_{1\times N} &  -K & -K & 0_{1\times N}\\C & 0_{1\times N} & 0_{1\times n} & 0_{1\times n} & 0_{1\times N}\end{bmatrix}.\label{eq:Cfb}
\end{align}
\begin{proposition}  The eigenvalues of $A_{fb}$ are the union of the eigenvalues of $A$, $A-BK$, $A-LC$, and $A_N$, with the latter having multiplicity two.
\end{proposition}
\begin{proof} Substituting \eqref{eq:aug_defs} and \eqref{eq:AC} into $A_{fb}$ yields a block $5\times 5$ matrix. Adding column 2 to column 5 and column 1 to column 4 and subtracting row 4 from row 1 and row 5 from row 2 on $A_{fb}-\lambda I_{3n+2N}$ yields a lower block triangular matrix with the same determinant $ \det (A_{fb}-\lambda I_{3n+2N}) = \det (A-\lambda I_n) \det(A-BK-\lambda I_n)\det(A-LC-\lambda I_n)\det (A_N-\lambda I_N)^2$.
\end{proof}
In the event that $A$ has an eigenvalue $\lambda = 0$ with multiplicity $m\geq 1$, then there exists a realization $(A_{fb}^*, B_{fb}^*, C_{fb}^*)$ of the system in Fig.~\ref{fig:F_optimal} that has one less state than that given by \eqref{eq:Afb}-\eqref{eq:Cfb}. Specifically,  the multiplicity of $\lambda = 0$ as an eigenvalue  of $A^*_{fb}$ will be equal to $m-1$; for details see Lemma~\ref{prop:AcCc_evect0} and Proposition \ref{prop:AcCc_minreal} in Appendix~\ref{sec:continuous_outputdelay}.


\section{Relation To The Smith Predictor}
\label{sec:Smith}

The Smith predictor 
\begin{equation}\notag
C_p(s) = \frac{C(s)}{1+C(s)G(s)(1-e^{-s\tau})},
\end{equation}
shown in the dashed box in Fig.~\ref{fig:Smith}, is a commonly proposed   architecture for compensating systems with a time delay. This architecture has several key features:
\begin{enumerate}[S-(a)]
\item The predictor $C_p(s)$ contains a model of both the time delay as well as the plant $G(s)$.
\item The controller $C(s)$ may be designed to stabilize the plant $G(s)$, ignoring   the time delay, and the resulting command response is given by $Y(s)=G(s)C(s)/(1+G(s)C(s))$.
\item The only input to the compensator is the error signal, $e=r-w$, processed in a One Degree of Freedom (1DOF) control architecture. 
\item Only an input-output model of the plant is assumed, as opposed to an internal state variable description.
\item The poles of $G(s)$ appear as zeros of $C_p(s)$, and thus the
  Smith predictor cannot be used to stabilize unstable systems. The
  one exception to this rule is if $G(s)$ has a single pole at
  $s=0$. In that case $\lim_{s\rightarrow 0}C_p(s)\neq 0$, and it
  follows that the Smith predictor can be used to stabilize a plant
  with a single integrator.
\end{enumerate}

\begin{figure}[htbp]
\centerline{\epsfig{file=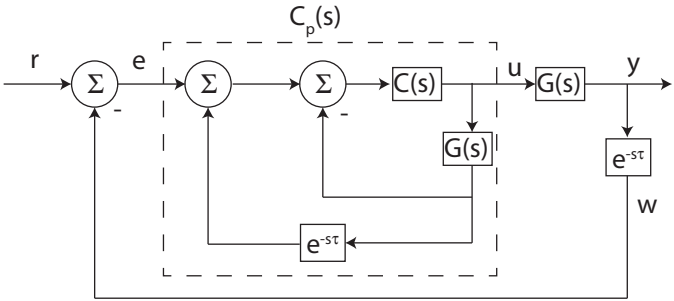,clip=,width=3.5in}}
\caption{Smith predictor architecture.}
\label{fig:Smith}
\end{figure} 

By way of contrast, the optimal control architecture shown in Fig.~\ref{fig:F_optimal} has the following corresponding properties:
\begin{enumerate}[O-(a)]
\item The compensator contains a copy of the plant $G(s)$ but only the approximation $P_N(s)$ to the delay.
\item The gains $K$ and $L$ are designed  ignoring the time delay. In the absence of noise and disturbances, the command response, given by \eqref{eq:YandDelta} with $L_1$ replaced by $L$ and $F(s)$ from \eqref{eq:Fsopt}, will depend on the error $\Delta(s)$ in approximating the time delay, but will be small if the system bandwidth  is confined to the frequency range for which $\Delta(j\omega)\approx 0$.
\item  The delayed output $w$ and the command input $r$ are processed separately using a Two Degree of Freedom (2DOF) control architecture.
\item An internal model of the plant is used as the basis for state feedback and optimal estimation.
\item The poles of $G(s)$ appear as zeros of $C_{obs}(s)$, and thus the architecture   in Fig.~\ref{fig:F_optimal} cannot be used to stabilize an unstable system. The one exception to this rule is if $G(s)$ has a single pole at $s=0$. In that case $C_{obs}(0)\neq 0$, and it follows that the architecture can be used to stabilize a plant with a single integrator.
\end{enumerate}

First consider properties (a) and (b). Taking the limit as $N\rightarrow \infty$ in allows us to replace the Pad\'e approximation $P_N(s)$ by an exact copy of the time delay. It follows that the approximation error $\Delta (s) = 0$ and the command response in \eqref{eq:YandDelta} no longer depends on the time delay. With regard to property (c), it is  possible to rearrange the block diagram in Fig.~\ref{fig:F_optimal} by subtracting the reference signal from the summing junction at the upper left, so that both systems have 1DOF control architectures. Note that, with no disturbance or noise present, and assuming perfect models of the plant and delay, the signal $e^m-v=-r$, and thus the system output $y$ will not depend on the delay despite the rearranged control architecture. The remaining difference between the two architectures, property (d), is needed in order to give more detailed modeling of the effect of the disturbance $d$ on the plant dynamics, to apply state estimate feedback, and to use optimal estimation to update these state estimates based on the output measurement.
\section{Discussion}
\label{sec:conclusions}

The cerebellum is critical for coordinating movement, balance, and
motor learning \cite{bastian2006learning}. How does the cerebellum
compensate for delays? It is proposed that the cerebellum could
function through the state estimation process to compensate the
delay~\cite{miall1996forward,paulin1989kalman,crevecoeur2019filtering}.
  
In this paper, we focus on the delay-masking problem in the state
estimation process, which must be understood to properly interpret
experimental data that investigates delay compensation. First we show
the delay is ``masked''---in the sense described by Carver et al.\ \cite{Carver09}---
during  delay compensation. That is, the feedback of estimated
plant states partially inverts the delay by introducing zeros at the
poles of the Pad\'e approximation (and zeros at the origin in the discrete-time
  case). This pole--zero cancellation was found to mask both the
sensory feedback delay and how it is regulated in the input--output
relationship from plant measurement $y$ to the computed motor commands
$u$. Second, we found that optimal estimation leads to a decomposition
of the Kalman filter; specifically, we found that the $N$ Pad\'e
states were uncontrollable while the remaining eigenvalues were
identical to those found by solving a lower order estimation problem
that ignores the time delay. The optimal reduction
identified above may give insights into how the brain processes the
sensory information for state estimation. Furthermore, the tradeoff filter in
estimation error, which modulates the balance between process
disturbance and measurement noise, could be obtained by ignoring the
time delay, as long as the Pad\'e approximation error is sufficiently
small.

We explore further how to exploit this
decomposition to produce structural simplifications in control.
Specifically, we found that an optimal observer-based compensator,
where only the plant (and not the delay) states are used for the
controller, can be subtly modified so that it acts surprisingly
similarly to a Smith predictor, sharing a number of a Smith
predictor's shortcomings. Specifically, both the Smith predictor and
our new ``Smith-like'' predictor fail to stabilize (most) unstable
plants.

It has been hypothesized that
the cerebellum might act like a Smith predictor for such behaviors as
reaching \cite{zimmetcerebellar2020,miall1993cerebellum}, but it
cannot be used to explain the cerebellum's role in postural balance
\cite{kiemelidentification2011}, because the latter involves
stabilization of right-half-plane poles. While there are many studies
working on generalizing the Smith predictor for cases such as unstable
plants\cite{sanz2018generalized}, the simple formulation we use in
this paper of an observer-based controller based on a Pad\'{e}
approximation can stabilize unstable plants and does not suffer the
drawbacks of a Smith predictor. Nevertheless, our calculations
suggest that Smith predictors serve as simplified (though limited)
near-optimal solutions to the delay compensation problem, and the
minimal modifications needed to transform a state-estimation-based
controller into a Smith predictor suggest that it may be infeasible to experimentally differentiate the two approaches, while at the same time suggesting
that either may be a useful model, depending on the goal of any given
study.

In our analysis we included the measurement noise $n$ after the time
delay; see Equation \eqref{eq:wm}. This is the equivalent condition as
adding noise before the time delay as follows:
\begin{equation}\label{eq:wm_beforedelay}
    w^m(t) = w(t)+n(t-\tau),
\end{equation} 
This works since the noise is serially uncorrelated, and Equations
\eqref{eq:wm} and \eqref{eq:wm_beforedelay} are statistically
equivalent.

Future work will include applying these results to multiple sensory
systems---such as haptic and visual feedback---with differing delays.
In addition, robustness is important for systems with time delays
\cite{zhong2006robust}, especially in the face of errors in delay
approximation. 
We are interested in exploring this and
other controller design issues as future work.


 \section*{ACKNOWLEDGMENT} 
 The authors would like to thank Pete Seiler for helpful discussions of the results in Section~\ref{sec:optimal}.  The work was motivated by
 discussions with Amanda Zimmet and Amy Bastian. This work is supported by the National Science Foundation (1825489 and 1825931), the National Institutes of Health under (R01-HD040289), and the Office of Naval Research (N00014-21-1-2431).

\begin{appendices}
\section{Supplementary Results for Continuous System with Output Delay}\label{sec:continuous_outputdelay}
\textbf{Proof of Lemma~\ref{prop:Cobs_minimal}:} 
 \begin{proof}
 $(i)$ We know that $\lambda$ is an uncontrollable eigenvalue of $(A_{obs}, L_{aug})$ if  and only if there exists a nonzero vector $\begin{bmatrix}x^T & w^T\end{bmatrix}$ such that 
 \begin{equation}\label{eq:PBH1}
   \begin{bmatrix}x^T & w^T\end{bmatrix}\begin{bmatrix}\lambda I_{n+N}-A_{obs}& L_{aug}\end{bmatrix}= 0.
 \end{equation}
 Substituting \eqref{eq:aug_defs} and the definitions of $K_{aug}$ and $L_{aug}$ into \eqref{eq:PBH1} yields 
 \begin{align}
 x^T(\lambda I_n-A+BK+L_1D_NC) + w^T(-B_NC+L_2D_NC) &=0\notag\\
 x^TL_1C_N+w^T(\lambda I_N-A_N+L_2C_N)&=0\notag\\
 x^TL_1+w^TL_2&=0,\label{eq:PBHc3}
 \end{align}
 and using \eqref{eq:PBHc3} implies
 \begin{align}
 x^T(\lambda I_n-A+BK) + w^T(-B_NC) &=0\label{eq:PBHc1a}\\
  w^T(\lambda I_N-A_N)&=0,\notag
 \end{align}
 from which it follows that either $w=0$ or that $\lambda$ must be an eigenvalue of $A_N$. In the former case \eqref{eq:PBHc1a} implies that $\lambda$ must be an eigenvalue of $A-BK$ that is uncontrollable from $L_1$, which is ruled out by assumption. In the latter case, it follows from \eqref{eq:PBHc1a} that $x^T = w^TB_NC(\lambda I_n-A+BK)^{-1}$, and from \eqref{eq:PBHc3} that 
 \begin{equation}\label{eq:Lcond}
    w^TB_NC(\lambda I_n-A+BK)^{-1}L_1 + w^TL_2 = 0.
 \end{equation}
 Since \eqref{eq:Lcond} must hold for each eigenvalue of $A_N$, the result follows.
 $(ii)$ We know that $\lambda$ is an unobservable eigenvalue of $(A_{obs}, K_{aug})$ if and only if there exists a nonzero vector $\begin{bmatrix}x^T & w^T\end{bmatrix}^T$ such that 
 \begin{equation}\label{eq:PBH2}
   \begin{bmatrix}\lambda I_{n+N}-A_{obs}\\ K_{aug}\end{bmatrix}\begin{bmatrix}x \\ w\end{bmatrix}= 0.
 \end{equation}
 Using \eqref{eq:aug_defs}   in \eqref{eq:PBH2} yields  
 \begin{align}
     (\lambda I_n-A+BK+L_1D_NC)x + L_1C_Nw &= 0\label{eq:PBHo1}\\
     (-B_NC+L_2D_NC)x  +  (\lambda I_N-A_N+L_2C_N)w &= 0\label{eq:PBHo2}\\
     Kx &= 0.\label{eq:PBHo3}
 \end{align}
 Note that if $x=0$ then \eqref{eq:PBHo1}-\eqref{eq:PBHo2} imply that  $\lambda$ must be an unobservable eigenvalue of $A_N$, which is ruled out by the assumption that \eqref{eq:qdot} is minimal. Using \eqref{eq:PBHo3} in \eqref{eq:PBHo1} yields
 \begin{equation}\label{eq:PBHo1a}
        (\lambda I_n-A+L_1D_NC)x + L_1C_Nw = 0
 \end{equation}
 and it follows from \eqref{eq:PBHo1a} and \eqref{eq:PBHo2} that $\lambda$ must be an eigenvalue of $A_{aug}-L_{aug}C_{aug}$. By assumption $\lambda$ is not an eigenvalue of $A-L_1D_NC$, and thus \eqref{eq:PBHo1a} and \eqref{eq:PBHo3} imply that 
 \begin{equation}\notag
     K(\lambda I_n-A+L_1D_NC)^{-1}L_1C_Nw= 0,
 \end{equation}
 which is ruled out by hypothesis, and the result follows.
 \end{proof}


\textbf{Proof of Lemma~\ref{prop:Cobsnewmin}:}
 \begin{proof}
  $(i)$ We know that $\lambda$ is an uncontrollable eigenvalue of $(A_C, B_C)$ if and only if there exists a nonzero vector $\begin{bmatrix}x^T & w^T & u^T\end{bmatrix}$ such that 
 \begin{equation}\label{eq:PBHc}
     \begin{bmatrix}x^T & w^T & u^T\end{bmatrix}\begin{bmatrix}\lambda I_{2n+N}-A_C & B_C\end{bmatrix} = 0.
 \end{equation}
 Substituting \eqref{eq:AC}  into \eqref{eq:PBHc} and rearranging yields $x^TL = 0$ and 
 \begin{align}
 x^T(\lambda I_n-A) + w^T BK &= 0\label{eq:c1r}\\
  w^T(\lambda I_n-A+BK)-u^TB_NC &= 0\label{eq:c2r}\\
  u^T(\lambda I_N-A_N) &= 0\label{eq:c3r},
 \end{align}
 and thus \eqref{eq:c3r} $\lambda$ must be an eigenvalue of $A_N$ or $u^T=0$. 
 
 In the latter case, it follows from \eqref{eq:c2r} that $\lambda$ must be an eigenvalue of $A-BK$ and from \eqref{eq:c1r} and $x^TL = 0$ that $w^TBK(\lambda I_n-A)^{-1}L=0$. However, our assumptions imply that $w^TB\neq 0$ (otherwise $\lambda$ would be an uncontrollable eigenvalue of $A$) and that $K(\lambda I_n-A)^{-1}L\neq 0$.
 
 In the former case, \eqref{eq:c1r} and \eqref{eq:c2r} imply that 
 \begin{align}
 x^T &= -w^TBK(\lambda I_n-A)^{-1}\label{eq:c1s}\\
 w^T &= u^TB_NC(\lambda I_n-A+BK)^{-1},\label{eq:c2s}
 \end{align}
 and substituting \eqref{eq:c2s} into \eqref{eq:c1s} and invoking $x^TL = 0$ yields
 \begin{equation}\label{eq:c5}
    u^TB_NC(\lambda I_n-A+BK)^{-1}BK(\lambda I_n -A)^{-1}L=0.
 \end{equation}
 The assumption that  $(A_N, B_N)$ is minimal implies that $u^TB_N\neq 0$, and $K(\lambda I_n -A)^{-1}L$ is also nonzero by assumption. The identity $C(\lambda I_n-A+BK)^{-1}B=G(s)/(1+K(sI_N-A)^{-1}B)$ together with the assumption that $G(\lambda)\neq 0$ imply that \eqref{eq:c5} cannot hold and thus that $(A_C, B_C)$ is controllable.

 $(ii)$   We know that $\lambda$ is an unobservable eigenvalue of $(A_C, C_C)$ if and only if there exists a nonzero vector 
 $\begin{bmatrix}x^T & v^T & w^T\end{bmatrix}^T$ such that
 \begin{equation}\label{eq:PBHo}
   \begin{bmatrix}\lambda I_{2n+N}  -A_C\\C_C\end{bmatrix}\begin{bmatrix}x \\ v \\ w\end{bmatrix} = 0.
  \end{equation}
  Substituting \eqref{eq:AC}  into \eqref{eq:PBHo} yields the four equations
  \begin{align}
  (\lambda I_n-A+LC)x + LD_NCv + LC_Nw &= 0\label{eq:o1}\\
  BKx + (\lambda I_n-A+BK)v &= 0\label{eq:o2}\\
  -B_NCv + (\lambda I_N-A_N)w &= 0\label{eq:o3}\\
  Kx+Kv &= 0.\label{eq:o4}
  \end{align}
  If $A-LC$ has an eigenvector $x$ satisfying $Kx=0$ then a solution to \eqref{eq:PBHo} is obtained by setting $v$ and $w$ equal to $0$;   this scenario is ruled out by the hypothesis that $(A-LC,K)$ is observable.  Hence, substituting \eqref{eq:o4} into \eqref{eq:o2} yields $(\lambda I_n-A)v=0$ and thus $\lambda$ must be an eigenvalue of $A$. It follows from \eqref{eq:o3} that  $w = (\lambda I_N-A_N)^{-1}B_NCv$,
  and substituting into \eqref{eq:o1} and rearranging yields
 $(\lambda I_n-A+LC)x+LP_N(\lambda)Cv = 0$.
  The assumption that $\lambda$ is not an eigenvalue of $A-LC$  implies that
 $x = -(\lambda I_n-A+LC)^{-1}LP_N(\lambda)Cv$, and 
 the fact that $v$ is an eigenvector of $A$ implies that
 $x+v = (\lambda I_n-A+LC)^{-1}(LCw-LP_N(\lambda)Cv)$,
 which is equal to zero if and only if $\lambda = 0$ so that $P_N(0)=1$, and the result follows.
 \end{proof}

 \begin{lemma}\label{prop:AcCc_evect0}Assume that $A$ has an eigenvalue $\lambda=0$ with eigenvector $v$. Then $(A_C, C_C)$ has an unobservable eigenvalue $\lambda=0$ with eigenvector $v_C = \begin{bmatrix}v^T&-v^T&(A_N^{-1}B_Nv)^T\end{bmatrix}^T$.
 \end{lemma}
 \begin{proof}The result follows immediately from the proof of Lemma~\ref{prop:Cobsnewmin}\eqref{lem:ptii}.
 \end{proof}
 
 Define   $\zeta = A_N^{-1}B_NCv$ and $\Theta = \left(I_N-(\zeta-e_N)e_N^T/e_N^T\zeta\right)$,
 and let $A_{N\times N-1}$, $C_{N\times N-1}$, and $I_{N-1\times N}$ denote the first $N-1$ columns of $A_N$, $C_N$, and the first $N-1$ rows of $I_N$, respectively.
 
 \begin{proposition}\label{prop:AcCc_minreal}
    Assume that $A$ has an eigenvalue $\lambda = 0$ with multiplicity $m\geq 1$.  Then a minimal realization of $C_{obs}(s)$ defined in \eqref{eq:xCdot} is given by
    \begin{align*}
    &A_C^* = \notag\\ &\begin{bmatrix}A-LC & -LD_NC & A_C^*(1,3)\\
                         -BK & A^*_C(2,2)&v\zeta^{-L}A_{N\times N-1}\\
                         0_{N-1\times n} &I_{N-1\times N}\Theta B_NC &I_{N-1\times N}\Theta A_{N\times N-1} \end{bmatrix},\\ 
                         &B_C^* = \begin{bmatrix}L\\0_{n+N-1\times 1}\end{bmatrix}, \quad
    C_C^*  = \begin{bmatrix}K & K & 0_{1,N-1}\end{bmatrix},
    \end{align*}
 where    $A_C^*(1,3) =-LC_{N\times N-1}-v\zeta^{-L}A_{N\times N-1}$ and $A^*_C(2,2) =A-BK+v\zeta^{-L}B_NC$. 
 
    A minimal realization of the feedback system \eqref{eq:xfb}-\eqref{eq:Cfb}  is given by $(A_{fb}^*,  B_{fb}^*,  C_{fb}^*)$, where
 \begin{equation}\notag
   A^*_{fb} = \begin{bmatrix}A_{aug} & -B_{aug}C_C^*\\B_C^*C_{aug} & A_C^*\end{bmatrix},  
   \end{equation}
 and  $B_{fb}^*$ and $C_{fb}^*$   are obtained by deleting the last row of  $B_{fb}$  and    the last column of  $C_{fb}$ .
 \end{proposition}
 \begin{proof}Define
 \begin{equation}\notag
 Q = \begin{bmatrix}I_n&0_n&0_{n\times N-1}&v\\
                         0_n&I_n&0_{n\times N-1}&-v\\
                         0_{N\times n}& 0_{N\times n}& I_{N\times N-1} & \zeta\end{bmatrix}, 
 \end{equation}
 and $x^\dagger_C = Q^{-1}x_C$, $A^\dagger_C = Q^{-1}A_CQ$, $B^\dagger_C = Q^{-1}B_C$, and $C^\dagger_C = C_CQ$.
 Then $B^\dagger_C = B_C$, $C^\dagger_C = C_C$, and 
 \begin{align}\notag
 &A^\dagger_C  =   &\begin{bmatrix}A-LC & -LD_NC & A^*_C(1,3)& 0_{n\times 1}\\
                         -BK & A^*_C(2,2)&v\zeta^{-L}A_{N\times N-1}&0_{n\times 1}\\
                         0_{N\times n} &\Theta B_NC &\Theta A_{N\times N-1} & 0_{N\times 1} \end{bmatrix}.
  \end{align}
 It is clear by inspection that $A^\dagger_C$ has an unobservable eigenvalue $\lambda=0$  corresponding to state $x_C^\dagger(2n+N)$. Deleting this state yields the lower order realization $(A_C^*, B_C^*, C_C^*)$. Replacing $(A_C, B_C,  C_C)$ in $A_{fb}$ with $(A_C^\dagger, B_C^\dagger, C_C^\dagger)$ also yields an unobservable eigenvalue, and deleting the corresponding state yields $(A_{fb}^*,  B_{fb}^*,  C_{fb}^*)$.
  \end{proof}
\renewcommand{\thetheorem}{B.\arabic{theorem}}
\section{Optimal Lower Order Property for Continuous System with Input Delay}
 \label{sec:continuous_inputdelay}
The lower order property of optimal state feedback design for the continuous-time system with input delay is dual to the output delay condition as above. Since the proof follows similarly, we state the results below without proof and point out their dual relationship to the output delay case. 

Consider the single input, single output linear system
 \begin{equation}
     \dot{x} = \bar{A}x + \bar{B}w(t)+ \bar{E}d, \qquad y = \bar{C}x, \qquad x\in \mathbb{R}^n, \label{eq:xdot_inputdelay}
 \end{equation}
Suppose that the control input is delayed by $\tau$ seconds, so that only the delayed input
 \begin{equation}\notag
    w(t) = u(t-\tau)
 \end{equation}
  is available to the controller. And denote the transfer function from $w(t)$ to $y$ by $G(s) = \bar{C}(sI_n-\bar{A})^{-1}\bar{B}$ and that from $d$ to $y$ by $G_d(s)=\bar{C}(sI_n-\bar{A})^{-1}\bar{E}$. Assume that $(\bar{A},\bar{C})$ is observable and that $(\bar{A},\bar{B})$ and $(\bar{A}, \bar{E})$ are controllable. 
 Assume the presence of additive measurement noise
 \begin{equation}\label{eq:ym}
    y^m = y+n.
 \end{equation}
 
 To obtain a finite dimensional system, we will approximate the time delay $e^{-s\tau}$ by passing $u(t)$ through an $N$'th order Pad\'e approximation with minimal realization
  \begin{equation} \label{eq:qdot_CI}
  \dot{q}_N = \bar{A}_N q_N + \bar{B}_N u, \;w_N = \bar{C}_N q_N+\bar{D}_N u, \;q_N\in \mathbb{R}^N       
 \end{equation}
 and transfer function $P_N(s)=\bar{C}_N (sI_N-\bar{A}_N )^{-1}\bar{B}_N +\bar{D}_N$. 

Denote the system obtained by augmenting the Pad\'e state equations \eqref{eq:qdot_CI} to those of the plant \eqref{eq:xdot_inputdelay} with noisy measurement \eqref{eq:ym}  by
 \begin{align}
     \dot{x}_{aug} &= \bar{A}_{aug}x_{aug} + \bar{B}_{aug}u +  \bar{E}_{aug}d,  \label{eq:xaugdot_inputdelay}\\
              y^m_N &= \bar{C}_{aug}x_{aug} +n,    \label{eq:yN1}
 \end{align}
 where $x_{aug} = \begin{bmatrix}x^T & q_N^T\end{bmatrix}^T$,
 \begin{equation}\label{eq:augu_defs}
 \bar{A}_{aug} = \begin{bmatrix}\bar{A} & \bar{B}\bar{C}_N \\ 0_{N\times n} & \bar{A}_N\end{bmatrix}, \; \bar{B}_{aug} = \begin{bmatrix}\bar{B}\bar{D}_N \\ \bar{B}_N\end{bmatrix},
 \end{equation}	
 $\bar{E}_{aug} = \begin{bmatrix}\bar{E}^T & 0_{N\times 1}\end{bmatrix}^T$ and $\bar{C}_{aug} = \begin{bmatrix}\bar{C} & 0\end{bmatrix}$.
It is straightforward to show that if $G(s)$ has no zeros at the eigenvalues of $A_N$, then $(\bar{A}_{aug}, \bar{C}_{aug})$ is observable. Similarly, if $P_N(s)$ has no zeros at the eigenvalues of $A$, then $(\bar{A}_{aug}, \bar{B}_{aug})$ is controllable.

 Let the control law be given by state estimate feedback 
 \begin{equation} \notag
     u = -K_{aug}x_{aug}+Hr
 \end{equation}
$K_{aug}$ is the optimal feedback gain given by LQR design as $K_{aug}=R^{-1}\bar{B}_{aug}^{T}\Sigma_{aug}$, and $\Sigma_{aug}$ is the solution to $(n+N)$ dimensional Riccati equation
 \begin{equation} \notag
     \bar{A}_{aug}^{T}\Sigma_{aug} + \Sigma_{aug}\bar{A}_{aug}  +Q_{aug} - \Sigma_{aug}\bar{B}_{aug}R^{-1}\bar{B}_{aug}^{T}\Sigma_{aug} = 0
 \end{equation}
 where $Q_{aug} = \begin{bmatrix}Q & 0 \\ 0 & 0\end{bmatrix}$ and $Q$ is the cost on plant states. We now characterize the $N+n$ eigenvalues of state feedback $\bar{A}_{aug}-\bar{B}_{aug}K_{aug}$.
 \begin{proposition}\label{prop:Continuous_opteval} Define   $K=R^{-1}\bar{B}^{T}\Sigma$, $\Sigma$ is the solution to $n$ dimensional Riccati equation
 \begin{align} \notag
     \bar{A}^{T}\Sigma + \Sigma \bar{A} + Q - \Sigma \bar{B}R^{-1}\bar{B}^{T}\Sigma = 0
 \end{align} 
 Assume that the eigenvalues of $\bar{A}_N$, $\bar{A}$,  and $\bar{A}-\bar{B}K$ are disjoint. Then $\bar{A}_{aug}-\bar{B}_{aug}K_{aug}$ has
 \begin{enumerate}[(i)]
 \item $N$ eigenvalues identical to those of $\bar{A}_N$, and
 \item $n$ eigenvalues identical to those  of $\bar{A}-\bar{B}K$. 
\end{enumerate}

 \end{proposition}
\begin{proof}
  Note that $(A_N,B_N,C_N,D_N)$ is an $N$'th order Pad\'e approximation as in the main manuscript, then $(\bar{A}_N,\bar{B}_N,\bar{C}_N,\bar{D}_N) = (A_N^{T},C_N^{T},B_N^{T},D_N^{T})$ is also an $N$'th order Pad\'e approximation. Consider the plant dynamics $(\bar{A},\bar{B},\bar{C}) = (A^{T},C^{T},B^{T})$, which are dual to the output delay condition as above. Then the augmented system would also show duality as $(\bar{A}_{aug},\bar{B}_{aug},\bar{C}_{aug}) = (A_{aug}^{T},C_{aug}^{T},B_{aug}^{T})$. $Q_{aug}$ has a block structure similar as $V_{aug}$ in \eqref{eq:Ric_aug}. Essential to the proof are the facts that the cost only takes into account the plant states and does not penalize the delay states. So this partial pole placement property follows. 
 \end{proof}
 
 \section{Discrete System: Feedback of Estimated Plant States Partially Inverts Delay}
 \label{sec:discrete_feedback}
 \renewcommand{\thetheorem}{C.\arabic{theorem}}

Carver et al.~\cite{Carver09} showed that if the dynamics of a sensor were included in the model of the system used for estimator design, but that the state-estimate feedback only included nonzero gains for the plant states, then the transfer function of the resulting observer-based compensator would place zeros at sensor poles. While their proofs were performed for continuous-time systems, their results suggest that, in the discrete-time setting, a one-step delay in the sensor (which introduces a pole at the origin) will be canceled by a zero at the origin in the observer-based compensator (under similar assumptions as in the present paper). Their result does not trivially extend to the multi-step delay case (which leads to repeated poles at $z=0$). In this appendix, we generalize to the multi-step delay.
 
 Consider the single input, single output  linear system
 \begin{align}
 x_{k+1} &= Ax_k+Bu_k + Ed_k, \qquad x\in \mathbb{R}^n,   \label{eq:xkp1_dis}\\
    y_k &= Cx_k \notag
 \end{align}
 and define the plant transfer function $G(z)=C(zI_n-A)^{-1}B$. Assume there exists an $M$-step delay in the measurement of the output,
 \begin{equation}\label{eq:wk_dis}
    w_k = y_{k-M},
 \end{equation}
 as well as additive measurement noise
 \begin{equation}\notag
    w_k^m = w_k + n_k.
    \end{equation}
 For later reference, we adopt a state variable model of the delay, where
 \begin{align}
 q_{k+1}&=A_qq_k + B_qy_k, \qquad q\in \mathbb{R}^M,  \notag\\
            w_k &= C_qq_k, \notag
 \end{align}
$z^{-M} = C_q(zI_M-A_q)^{-1}B_q$, and
 \begin{align*}
    A_q & = \begin{bmatrix}0 & \cdots &0&0\\
    1&0&\cdots&0\\\vdots & \ddots & \ddots   &\vdots \\
                                       0&   \cdots &1 &0\end{bmatrix}_{M\times M}, \quad B_q = \begin{bmatrix}1\\0\\\vdots\\0\end{bmatrix}_{M\times 1}\\
    C_q & = \begin{bmatrix}0&\cdots&0&1\end{bmatrix}_{1\times M}.
 \end{align*}

 The control law is given by state estimate feedback,
 \begin{equation}\label{eq:uk_dis}
    u_k = -K\hat{x}_{k|k-1} + Hr_k,
 \end{equation}
 where the state estimates must be obtained from the delayed measurement of the output $y_k$. Hence the observer must estimate both the plant states and the delay states:
 \begin{equation}\label{eq:state_eqns_dis}
   \begin{bmatrix}\hat{x}_{k+1|k}\\\hat{q}_{k+1|k} \end{bmatrix} = \begin{bmatrix}A & 0\\B_qC & A_q\end{bmatrix}  \begin{bmatrix}\hat{x}_{k|k-1}\\\hat{q}_{k|k-1} \end{bmatrix} +\begin{bmatrix}B\\0\end{bmatrix}u_k  +\begin{bmatrix}L_1\\L_2\end{bmatrix}(w_k^m-\hat{w}_{k|k-1}).
 \end{equation}
 Substituting the control law \eqref{eq:uk_dis} into \eqref{eq:state_eqns_dis} yields
 \begin{align}\label{eq:state_eqns1_dis}
  \begin{bmatrix}\hat{x}_{k+1|k}\\\hat{q}_{k+1|k} \end{bmatrix} &= \begin{bmatrix}A-BK & -L_1C_q\\B_qC &A_q- L_2C_q\end{bmatrix}  \begin{bmatrix}\hat{x}_{k|k-1}\\\hat{q}_{k|k-1} \end{bmatrix} \notag\\
  &+ \begin{bmatrix}L_1\\L_2\end{bmatrix}w_k^m + \begin{bmatrix}BH\\0\end{bmatrix}r_k.
 \end{align}
 Denote the transfer function of the   observer based compensator mapping  $w_k^m$ to $-u_k$ by $C_{obs}(z)$. It follows from \eqref{eq:state_eqns1_dis} that $C_{obs}(z)$ has the state variable description 
 \begin{align}
  \begin{bmatrix}\hat{x}_{k+1|k}\\\hat{q}_{k+1|k} \end{bmatrix} &= \begin{bmatrix}A-BK & -L_1C_q\\B_qC &A_q- L_2C_q\end{bmatrix}  \begin{bmatrix}\hat{x}_{k|k-1}\\\hat{q}_{k|k-1} \end{bmatrix} + \begin{bmatrix}L_1\\L_2\end{bmatrix}w_k^m,\label{eq:Cobs_dis}\\
            - u_k &= \begin{bmatrix}K&0\end{bmatrix}\begin{bmatrix}\hat{x}_{k|k-1}\\\hat{q}_{k|k-1} \end{bmatrix}.\label{eq:ukbig_dis}
 \end{align}
 The Rosenbrock System Matrix associated with \eqref{eq:Cobs_dis}-\eqref{eq:ukbig_dis} is given by
 \begin{equation}\label{eq:RSM_dis}
 RSM(z) = \begin{bmatrix}zI_n-A+BK&L_1C_q&-L_1\\-B_qC&zI_M-A_q+L_2C_q&-L_2\\K&0&0\end{bmatrix},
 \end{equation}
 Let $0_{n\times 1}$ denote an $n$ dimensional column vector of zeros, and let $e_1$ denote the first standard basis vector. Then it may be verified from \eqref{eq:RSM_dis} and the structure of $A_q$ and $C_q$ that $RSM(0)$ has a nontrivial nullspace spanned by the vector   
 \begin{equation*}
 \begin{bmatrix}0_{n\times1}  \\ e_1 \\ 1
 \end{bmatrix},  
 \end{equation*}
 and thus that $C_{obs}(z)$ has at least one zero  at $z=0$. 
Indeed, the transmission zeros of $C_{obs}(z)$ include zeros at $z=0$ with multiplicity $M$. To do so, assume $M>1$ and use the Rosenbrock matrix and induction to show that $C_{obs}(z)/z,\ldots,C_{obs}(z)/z^{M-1}$ each have zeros at $z=0$, but that $C_{obs}(z)/z^M$ has no such zero. 
     A block diagram of the feedback system with  compensator $C_{obs}(z)$ is in Fig.~\ref{fig:system}; the dashed box in this diagram contains $-C_{obs}(z)$.       The additional zeros at $z=0$ in $C_{obs}(z)$ will lower its relative degree and thus reduce the delay in its impulse response, partially compensating for the effect of the time delay in the feedback path.

 \begin{figure}[htbp]
 \centerline{\epsfig{file=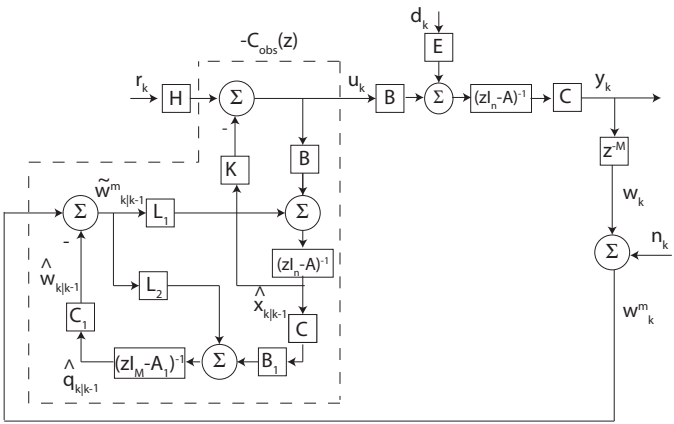,clip=,width=3.5in}}
 \caption{Discrete-time feedback system with $M$-step delay.}
 \label{fig:system}
 \end{figure}
 
Define the \emph{measured} estimation error for the delayed output by
 \begin{equation}\label{eq:wkmtilde_dis}
  \tilde{w}_{k}^m = w_k^m - \hat{w}_{k}.
 \end{equation} and the \emph{actual} estimation error for the delayed output by
 \begin{equation}\label{eq:wktilde_dis}
      \tilde{w}_k = w_k-\hat{w}_k.
 \end{equation}


\begin{corollary}\label{cor:F_dis}
    Define the \emph{tradeoff filter}
  \begin{equation} \label{eq:F_dis} F(z) =
    \left(1+C_{aug }\left(zI_{n+N}-A_{aug }\right)^{-1}L_{aug }\right)^{-1}.
  \end{equation} The $z$-transform of the estimation error \eqref{eq:wktilde_dis}  satisfies
    \begin{equation}\label{eq:F_tradeoff_dis}
   \tilde{W}(z)          = F(z) z^{-M}C(zI_n-A)^{-1}ED(z)  + \left(F(z)-1\right)N(z).
 \end{equation}
 \end{corollary}
 \begin{proof}
First, \eqref{eq:xkp1_dis}-\eqref{eq:wk_dis} imply that
 \begin{equation}\label{eq:W(z)_dis}
 W(z)=z^{-M}G(z)U(z) + z^{-M}C(zI_n-A)^{-1}ED(z).
 \end{equation}
 Taking $z$-transforms and applying \eqref{eq:uk_dis}  yields
 \begin{equation}\label{eq:Xhat_z_dis}
      \begin{aligned}
  \hat{X}(z) &= (zI_n-A)^{-1}BU(z)+(zI_n-A)^{-1}L_1\tilde{W}^m(z)\\
                 &= \left(I+(zI_n-A)^{-1}BK\right)^{-1}(zI_n-A)^{-1}\\
                 &\qquad \times {\left(L_1\tilde{W}^m(z)+BHR(z)\right)}.
 \end{aligned}
  \end{equation}
 Next, it follows from \eqref{eq:Cobs_dis} that 
 \begin{align}\notag
     \hat{q}_{k+1|k} &= B_qC\hat{x}_{k|k-1}+A_q\hat{q}_{k|k-1}+L_2\tilde{w}_{k|k-1}^m,\\
                  \hat{w}_{k+1|k} &= C_q\hat{q}_{k+1|k},\notag
 \end{align}
 which implies
 \begin{equation}\label{eq:What_z_dis}
    \hat{W}(z) = C_q\left(zI_M-A_q\right)^{-1}B_qC\hat{X}(z)+C_q\left(zI_M-A_q\right)^{-1}L_2\tilde{W}^m(z).
 \end{equation}
 Together, \eqref{eq:wkmtilde_dis} and \eqref{eq:What_z_dis} imply
 \begin{equation}\notag
    \tilde{W}^m(z)=W^m(z)-C_q\left(zI_M-A_q\right)^{-1}\left(B_qC\hat{X}(z)+L_2\tilde{W}^m(z)\right).
 \end{equation}
 Substituting \eqref{eq:Xhat_z_dis} yields
 \begin{align}
 \tilde{W}^m(z)&=\left(1+z^{-M}C(zI_n-A)^{-1}L_1  +C_q\left(zI_M-A_q\right)^{-1}L_2\right)^{-1}\notag\\
 &\qquad \times \left(W^m(z)-z^{-M}G(z)U(z)\right)\notag\\
      &= F(z)\left(W^m(z)-z^{-M}G(z)U(z)\right),     \label{eq:Wmtilde_dis}
 \end{align}
 Next, noting that $\tilde{w}_k^m = \tilde{w}+n_k$, it follows from \eqref{eq:Wmtilde_dis} that
 \begin{equation}\label{eq:Wtilde_dis}
   \tilde{W}(z)  = F(z) \left(W(z) +N(z)-z^{-M}G(z)U(z)\right) -N(z).
   \end{equation}
   Substituting \eqref{eq:W(z)_dis} into \eqref{eq:Wtilde_dis} yields \eqref{eq:F_tradeoff_dis}.
 \end{proof}
 It follows from Corollary~\ref{cor:F_dis} that $F(z)$ describes the tradeoff between the response of the estimation error to the plant disturbance $D(z)$ and the measurement noise $N(z)$. Using large estimator gains $L_1$ and/or $L_2$ will force $F(z)\approx 0$ and the disturbance response will be small at the expense of the noise being passed directly to the estimation error. Using small estimator gains has the opposite effect.
 
 Let us now calculate the response of the system output $y_k$.  The definition of $u_k$ and \eqref{eq:Xhat_z_dis} yield 
 \begin{align}\label{eq:U(z)_dis}
    U(z) &= -K(zI_n-A+BK)^{-1}L_1\tilde{W}^m(z) \notag \\
    & + \left(1+K(sI_n-A)^{-1}B\right)^{-1}HR(z).
 \end{align}
 It follows from \eqref{eq:xkp1_dis} and \eqref{eq:U(z)_dis} that
 \begin{align}
    Y(z)&= C(zI_n-A+BK)^{-1}BHR(z)+ C(zI_n-A)^{-1}ED(z)\notag\\
    &- C(zI_n-A)^{-1}BK(zI_n-A+BK)^{-1}L_1\tilde{W}^m(z),\notag
 \end{align}
 and from \eqref{eq:Wtilde_dis} that 
 \begin{equation}\notag
     \tilde{W}^m(z) =    F(z) z^{-M}C(zI_n-A)^{-1}ED(z)  + F(z)N(z).
 \end{equation}
Hence in the absence of disturbances and measurement noise the response of $y_k$ to a command $r_k$ is the same with the observer as with state feedback, \emph{even with the $M$-step delay in the feedback loop}.  
 \renewcommand{\thetheorem}{D.\arabic{theorem}}

 \section{Discrete System: Decomposition of Optimal Estimator with Delay}
  \label{sec:discrete_optimal}

We now characterize the closed loop eigenvalues for a discrete-time estimator with the structure given in Appendix~\ref{sec:discrete_feedback} when the estimator is optimal, and present counterparts to the results for continuous-time in Section~\ref{sec:optimal}. After doing so we will describe connections to the work of \cite{mirkin2021dead}, who study discrete-time systems with a delay at the plant input.

 To formulate the optimal estimation problem, we consider the augmented system 
 \begin{align}\label{eq:state_aug_eqns_dis}
 \hat{x}_{aug(k+1)} &= A_{aug}\hat{x}_{aug(k)} + B_{aug}u_k + L_{aug}(w_k^m-\hat{w}_{k|k-1})\\
 \hat{w}_{k|k-1} &= C_{aug}x_{aug(k)},\notag
 \end{align}
 where $\hat{x}_{aug(k)} = \begin{bmatrix}\hat{x}_{k|k-1}\\\hat{q}_{k|k-1} \end{bmatrix}, A_{aug} = \begin{bmatrix}A & 0\\B_qC & A_q\end{bmatrix},   B_{aug} = \begin{bmatrix}B\\0\end{bmatrix}, L_{aug} = \begin{bmatrix}L_1\\L_2\end{bmatrix}, C_{aug} = \begin{bmatrix}0 & C_q\end{bmatrix}$. 
 Suppose that $d_k$ and $n_k$ are zero mean Gaussian white noise processes with covariances $V\geq 0$ and $W>0$, respectively, and  assume that $(A_{aug},C_{aug})$ is observable and  $(A_{aug},E_{aug})$ is controllable. Then the optimal estimator gain satisfies
 \begin{equation}\label{eq:Laugopt_dis}
     L_{aug} = A_{aug}\Sigma_{aug} C_{aug}^T(C_{aug}\Sigma_{aug} C_{aug}^T+W)^{-1},
 \end{equation}
 where $\Sigma_{aug}$  is the unique positive semidefinite solution to the algebraic Riccati equation
 \begin{align}\label{eq:Ric_aug_dis}
&\Sigma_{aug} = A_{aug}\Sigma_{aug} A_{aug}^T +V_{aug} \\
&- A_{aug}\Sigma_{aug} C_{aug}^T(C_{aug}\Sigma_{aug} C_{aug}^T+W)^{-1}C_{aug} \Sigma_{aug} A_{aug}^T\notag
 \end{align}
 with $V_{aug} = E_{aug}VE_{aug}^T$. 
 
The eigenvalues of the optimal estimator are   those of the closed loop matrix $A_{aug}^{CL}=A_{aug}-L_{aug}C_{aug}$, and  will be characterized in Propositions~\ref{prop:eval_aug_lower}-\ref{zeroevals_unctrb} below, after the following lemma.
 Define  $L = A\Sigma C^T(C\Sigma C^T+W)^{-1}$, where $\Sigma$ is the unique positive semidefinite solution to the Riccati equation
  \begin{equation}\notag
      \Sigma = A\Sigma A^T - A\Sigma C^T(C\Sigma C^T+W)^{-1}C \Sigma A^T +EVE^T.
  \end{equation}
  \begin{lemma}\label{A_LC_lemma}
Suppose that $A$ has no eigenvalues equal to zero. Then   $A^{CL}=A-LC$ also has no eigenvalues equal to zero. If $A$ has an eigenvalue equal to zero with multiplicity $m\geq 1$, then $A^{CL}$ also has an eigenvalue equal to zero but with multiplicity one. Furthermore, this zero eigenvalue is an uncontrollable eigenvalue of $(A, L)$.
   \end{lemma}
 
 \begin{proof}
The results of \cite{pappas1980numerical} imply that the eigenvalues of $A^{CL}$ may be found from the generalized eigenvalue problem $M_\ell z_1=\lambda M_rz_1$, where $M_\ell = \begin{bmatrix}I_n&C^TW^{-1}C\\0_{n\times n}&A\end{bmatrix}$ and $M_r = \begin{bmatrix}A^T&0_{n\times n}\\-EVE^T&I_n\end{bmatrix}$. If $A$ has no zero eigenvalues, then it is clear that $M_\ell$ is nonsingular and thus zero cannot be an eigenvalue of $A^{CL}$. If $A$ has an eigenvalue equal to zero with multiplicity $m\geq 1$, then observability of $(A, C)$ implies that its geometric multiplicity must be equal to one. Hence there exist nonzero vectors $v_1, \ldots, v_m$ such that $Av_1=0$ and $Av_k=Av_{k-1}$, $k=2,\ldots,m$. Note that the vector $z_1=\begin{bmatrix}-C^TW^{-1}Cv_1\\v_1\end{bmatrix}$ lies in the nullspace of $M_\ell$ and thus $\lambda = 0$ is an eigenvalue of $A^{CL}$ with multiplicity at least one. For the multiplicity be greater than one, it is necessary that there exists $z_2=\begin{bmatrix}z_{21}\\z_{22}\end{bmatrix}$ such that $M_\ell z_2=M_r z_1$ from which it follows that 
$Az_{22}=EVE^TC^TW^{-1}Cv_1+v_1$.
 Properties of generalized eigenvectors imply that $v_1=Av_2$, and thus that 
 \begin{equation}\label{eq:Aeqn}
 A\left(z_{22}-v_2\right)=EVE^TC^TW^{-1}Cv_1.
 \end{equation}
 The fact that $(A, E)$ is controllable implies that $E$ does not lie in the columnspace of $A$, and thus that \eqref{eq:Aeqn} has no nonzero solution. It follows that $\lambda=0$ is an eigenvalue of $A^{CL}$ with multiplicity $m=1$.
The dual of Lemma~\ref{lem:FGK} states that if $(A,C)$ is observable and $A$ and $A-LC$ share eigenvalues at zero, then these zero eigenvalues are uncontrollable eigenvalues of $(A,L)$. 
 \end{proof}
 \begin{proposition}\label{prop:eval_aug_lower} Consider the optimal estimator defined by \eqref{eq:state_aug_eqns_dis}-\eqref{eq:Ric_aug_dis}. 
 Then $A^{CL}_{aug}=A_{aug}-L_{aug}C_{aug}$ has
 \begin{enumerate}[(i)]
 \item $M$ zero eigenvalues,
 \item $n$ eigenvalues identical to those  of $A^{CL}$.
\end{enumerate}

 \end{proposition}
 
  
 \begin{proof} 
 
To prove part~$(i)$, we note  the results of \cite{pappas1980numerical} imply that the optimal closed loop eigenvalues corresponding to the discrete Riccati equation \eqref{eq:Ric_aug_dis} may be found by solving the generalized eigenvalue problem 
  \begin{equation}\notag
    M_l z_1  = \lambda M_r  z_1,
 \end{equation}
 where    $M_l = \begin{bmatrix}I_{n+M} & C_{aug}^{T}W^{-1}C_{aug}\\0_{(n+M) \times (n+M)} & A_{aug}\end{bmatrix}$ and  $M_r = \begin{bmatrix}A_{aug}^{T} & 0_{(n+M) \times (n+M)}\\-V_{aug} & I_{n+M}\end{bmatrix}$. Using the structure of $A_q$ and $C_q$ it is straightforward to verify that 
 \begin{equation*}
     z_1 = \begin{bmatrix} 0_{(n+M-1) \times 1}^{T} &1&0_{(n+M-1) \times 1}^{T}&-W \end{bmatrix}^{T}
 \end{equation*}  lies in the nullspace of $M_\ell$, and thus that $\lambda = 0$ is an eigenvalue of $A_{aug}^{CL}$. Furthermore, it may be shown by   induction that the sequence of vectors 
 \begin{align*}
 z_2   &=  \begin{bmatrix} 0_{(n+M-2) \times 1}^{T} &1&0_{(n+M-1) \times 1}^{T}&-W&0 \end{bmatrix}^{T}\\ \ldots\\ z_M &= \begin{bmatrix} 0_{n \times1}^{T} & 1&0_{(n+M-1) \times 1}^{T}&-W&0_{(M-1) \times 1}^{T} \end{bmatrix}^{T}
 \end{align*}
 satisfies $M_\ell z_k=M_rz_{k-1},\;k=2,\ldots,M$ and thus, from the results of \cite{pappas1980numerical}, the multiplicity of $\lambda =0$ as an eigenvalue of $A_{aug}^{CL}$ must be at least $M$ due to the zeros introduced by the $M$-step delay. 
 
 
 To prove part~$(ii)$, let $\lambda$ be a nonzero eigenvalue of $A^{CL}$. Then there exists exists a nonzero vector $\begin{bmatrix}v_1^T&v_2^T\end{bmatrix}$ such that 
\begin{equation}\notag
   \begin{bmatrix}I_{n} & C^{T}W^{-1}C\\0_{n \times n} & A\end{bmatrix} \begin{bmatrix}v_1\\v_2 \end{bmatrix} = \lambda \begin{bmatrix}A^{T} & 0_{n \times n}\\-EVE^T & I_{n}\end{bmatrix}  \begin{bmatrix}v_1\\v_2 \end{bmatrix}.
 \end{equation}
 Define $a_1=-W^{-1}Cv_2\begin{bmatrix}\lambda^{-M}&\lambda^{-(M-1)}&\ldots&\lambda^{-1}\end{bmatrix}^T$, $a_2=-Wa_1$, and $\bar{v}=\begin{bmatrix}v_1^T&a_1^T&v_2^T&a_2^T\end{bmatrix}^T$. Then direct calculation shows that $M_\ell\bar{v}=\lambda M_r\bar{v}$, and thus $\lambda$ is also an eigenvalue of $A^{CL}_{aug}$.
 
It follows from Lemma~\ref{A_LC_lemma} that $A^{CL}$ can have an eigenvalue equal to zero with multiplicity at most one. In this case there exists a nonzero vector $\begin{bmatrix}v_1^T&v_2^T\end{bmatrix}$ such that \begin{equation}\notag
   \begin{bmatrix}I_{n} & C^{T}W^{-1}C\\0_{n \times n} & A\end{bmatrix} \begin{bmatrix}v_1\\v_2 \end{bmatrix} = 0_{2n \times 1}.
 \end{equation}
 Note next that $v_2$ must be nonzero, and define 
\begin{equation*}
z_{M+1}=\begin{bmatrix}C&0^T_{M\times 1}&\frac{-W}{Cv_2}v_2^T&0^T_{M\times 1}\end{bmatrix}^T.
\end{equation*}
Then $M_\ell z_{M+1}=M_rz_M$, and thus by the same reasoning as in the proof of part (i), $\lambda=0$ is an eigenvalue of $A^{CL}_{aug}$ with multiplicity $M+1$.

 \end{proof}
 
  \begin{proposition}\label{zeroevals_unctrb}
   The zero eigenvalues of $A^{CL}_{aug}$ are uncontrollable eigenvalues of $(A_{aug},L_{aug})$. 
  \end{proposition}
 
 \begin{proof}
 From the dual of Lemma~\ref{lem:FGK}, we see that since $(A_{aug},C_{aug})$ is observable and that $A_{aug}$ and $A^{CL}_{aug}$ share  zero eigenvalues with $A_{aug}$, then these zero eigenvalues are uncontrollable eigenvalues of $(A_{aug},L_{aug})$.  
 \end{proof}

 Let us now relate the preceding results to those of \cite{mirkin2021dead}. We showed in Proposition~\ref{zeroevals_unctrb} that an optimal estimation problem with a delay at the plant output results in closed loop estimator eigenvalues that are either at the origin or that arise from a lower order estimation problem that does not involve the delay. The authors of \cite{mirkin2021dead}, on the other hand, consider an optimal regulator problem with a delay at the plant input and show that the optimal eigenvalues are either at the origin or arise from a lower optimal regulator problem that ignores the delay. Although these are dual results, the proof techniques used in \cite{mirkin2021dead} are different than those above.

\end{appendices}

 
\bibliography{IEEEabrv,caodelay}

\begin{thebibliography}{10}
\providecommand{\url}[1]{#1}
\csname url@samestyle\endcsname
\providecommand{\newblock}{\relax}
\providecommand{\bibinfo}[2]{#2}
\providecommand{\BIBentrySTDinterwordspacing}{\spaceskip=0pt\relax}
\providecommand{\BIBentryALTinterwordstretchfactor}{4}
\providecommand{\BIBentryALTinterwordspacing}{\spaceskip=\fontdimen2\font plus
\BIBentryALTinterwordstretchfactor\fontdimen3\font minus
  \fontdimen4\font\relax}
\providecommand{\BIBforeignlanguage}[2]{{%
\expandafter\ifx\csname l@#1\endcsname\relax
\typeout{** WARNING: IEEEtran.bst: No hyphenation pattern has been}%
\typeout{** loaded for the language `#1'. Using the pattern for}%
\typeout{** the default language instead.}%
\else
\language=\csname l@#1\endcsname
\fi
#2}}
\providecommand{\BIBdecl}{\relax}
\BIBdecl

\bibitem{more2018scaling}
H.~L. More and J.~M. Donelan, ``Scaling of sensorimotor delays in terrestrial
  mammals,'' \emph{Proc R Soc B}, vol. 285, no. 1885, p. 20180613, 2018.

\bibitem{madhavsynergy2020}
M.~S. Madhav and N.~J. Cowan, ``The synergy between neuroscience and control
  theory: the nervous system as inspiration for hard control challenges,''
  \emph{Annu Rev Control Robot Auton Syst}, vol.~3, pp. 243--267, 2020.

\bibitem{franklin2008specificity}
D.~W. Franklin and D.~M. Wolpert, ``Specificity of reflex adaptation for
  task-relevant variability,'' \emph{J Neurosci}, vol.~28, no.~52, pp.
  14\,165--14\,175, 2008.

\bibitem{haith2016independence}
A.~M. Haith, J.~Pakpoor, and J.~W. Krakauer, ``Independence of movement
  preparation and movement initiation,'' \emph{J Neurosci}, vol.~36, no.~10,
  pp. 3007--3015, 2016.

\bibitem{zimmetcerebellar2020}
A.~M. Zimmet, D.~Cao, A.~J. Bastian, and N.~J. Cowan, ``Cerebellar patients
  have intact feedback control that can be leveraged to improve reaching,''
  \emph{eLife}, vol.~9, p. e53246, 2020.

\bibitem{susilaradeya2019extrinsic}
D.~Susilaradeya, W.~Xu, T.~M. Hall, F.~Galán, K.~Alter, and A.~Jackson,
  ``Extrinsic and intrinsic dynamics in movement intermittency,'' \emph{eLife},
  vol.~8, p. e40145, Apr. 2019.

\bibitem{crevecoeur2019filtering}
F.~Crevecoeur and M.~Gevers, ``Filtering compensation for delays and prediction
  errors during sensorimotor control,'' \emph{Neural Comput}, vol.~31, no.~4,
  pp. 738--764, 2019.

\bibitem{Smith57}
O.~J.~M. Smith, ``Closer control of loops with dead time,'' \emph{Chem. Eng.
  Prog.}, vol.~53, no.~5, pp. 217--219, 1957.

\bibitem{garcia2006control}
P.~Garcia, P.~Albertos, and T.~H{\"a}gglund, ``Control of unstable
  non-minimum-phase delayed systems,'' \emph{Journal of Process Control},
  vol.~16, no.~10, pp. 1099--1111, 2006.

\bibitem{sanz2018generalized}
R.~Sanz, P.~Garc{\'\i}a, and P.~Albertos, ``A generalized {S}mith predictor for
  unstable time-delay siso systems,'' \emph{ISA transactions}, vol.~72, pp.
  197--204, 2018.

\bibitem{lima2018robust}
B.~M. Lima, D.~M. Lima, and J.~E. Normey-Rico, ``A robust predictor for
  dead-time systems based on the {K}alman filter,'' \emph{IFAC-PapersOnLine},
  vol.~51, no.~25, pp. 24--29, 2018.

\bibitem{miall1993cerebellum}
R.~C. Miall, D.~J. Weir, D.~M. Wolpert, and J.~Stein, ``Is the cerebellum a
  {S}mith predictor?'' \emph{J Motor Behav}, vol.~25, no.~3, pp. 203--216,
  1993.

\bibitem{tolu2020cerebellum}
S.~Tolu, M.~C. Capolei, L.~Vannucci, C.~Laschi, E.~Falotico, and M.~V.
  Hern{\'a}ndez, ``A cerebellum-inspired learning approach for adaptive and
  anticipatory control,'' \emph{International journal of neural systems},
  vol.~30, no.~01, p. 1950028, 2020.

\bibitem{mirkin2003every}
L.~Mirkin and N.~Raskin, ``Every stabilizing dead-time controller has an
  observer--predictor-based structure,'' \emph{automatica}, vol.~39, no.~10,
  pp. 1747--1754, 2003.

\bibitem{mirkin2021dead}
L.~Mirkin and D.~Zanutto, ``Dead-time compensation as an observer-based
  design,'' \emph{IEEE Control Systems Letters}, vol.~6, pp. 1604--1609, 2021.

\bibitem{natori2012design}
K.~Natori, ``A design method of time-delay systems with communication
  disturbance observer by using {P}ad\'e approximation,'' in \emph{2012 12th
  IEEE International Workshop on Advanced Motion Control (AMC)}.\hskip 1em plus
  0.5em minus 0.4em\relax IEEE, 2012, pp. 1--6.

\bibitem{probst2010using}
A.~Probst, M.~Magana, and O.~Sawodny, ``Using a {K}alman filter and a {P}ad\'e
  approximation to estimate random time delays in a networked feedback control
  system,'' \emph{IET Control Theory Appl}, vol.~4, no.~11, pp. 2263--2272,
  2010.

\bibitem{Franklin}
G.~Franklin, J.~Powell, and A.~Emami-Naeini, \emph{Feedback Control of Dynamic
  Systems}, 5th~ed.\hskip 1em plus 0.5em minus 0.4em\relax Reading, Mass.:
  Addison--Wesley, 2006.

\bibitem{AM97}
P.~J. Antsaklis and A.~N. Michel, \emph{Linear Systems}.\hskip 1em plus 0.5em
  minus 0.4em\relax New York: McGraw--Hill, 1997.

\bibitem{Carver09}
S.~G. Carver, T.~Kiemel, N.~J. Cowan, and J.~J. Jeka, ``Optimal motor control
  may mask sensory dynamics,'' \emph{Biol Cybern}, vol. 101, no.~1, pp. 35--42,
  July 2009.

\bibitem{KS72}
H.~Kwakernaak and R.~Sivan, \emph{Linear Optimal Control Systems}.\hskip 1em
  plus 0.5em minus 0.4em\relax New York NY: Wiley-Interscience, 1972.

\bibitem{bastian2006learning}
A.~J. Bastian, ``Learning to predict the future: the cerebellum adapts
  feedforward movement control,'' \emph{Curr Opin Neurobiol}, vol.~16, no.~6,
  pp. 645--649, 2006.

\bibitem{miall1996forward}
R.~C. Miall and D.~M. Wolpert, ``Forward models for physiological motor
  control,'' \emph{Neural Netw}, vol.~9, no.~8, pp. 1265--1279, 1996.

\bibitem{paulin1989kalman}
M.~Paulin, ``A {K}alman filter theory of the cerebellum,'' in \emph{Dynamic
  interactions in neural networks: Models and data}.\hskip 1em plus 0.5em minus
  0.4em\relax Springer, 1989, pp. 239--259.

\bibitem{kiemelidentification2011}
T.~Kiemel, Y.~Zhang, and J.~J. Jeka, ``Identification of neural feedback for
  upright stance in humans: stabilization rather than sway minimization.''
  \emph{J Neurosci}, vol.~31, no.~42, pp. 15\,144--15\,153, Oct. 2011.

\bibitem{zhong2006robust}
Q.-C. Zhong, \emph{Robust control of time-delay systems}.\hskip 1em plus 0.5em
  minus 0.4em\relax Springer Science \& Business Media, 2006.

\bibitem{pappas1980numerical}
T.~Pappas, A.~Laub, and N.~Sandell, ``On the numerical solution of the
  discrete-time algebraic riccati equation,'' \emph{IEEE Transactions on
  Automatic Control}, vol.~25, no.~4, pp. 631--641, 1980.

\end{thebibliography}
\end{document}